\documentclass{gtpart}
\usepackage{pinlabel}
\usepackage{graphicx}
\title[Ordering Garside groups]{Ordering Garside groups}
\author[D Arcis]{$\,\,$Diego Arcis$^*$}
\givenname{Diego}
\surname{Arcis}
\address{IMB UMR 5584\\CNRS, Univ. Bourgogne Franche-Comt\'e\\21000 Dijon\\France}
\email{arcisd@gmail.com}	
\urladdr{}
\author[L Paris]{Luis Paris}
\givenname{Luis}
\surname{Paris}
\address{IMB, UMR 5584,CNRS, Univ. Bourgogne Franche-Comté, 21000 Dijon, France}
\email{lparis@u-bourgogne.fr}
\urladdr{}

\volumenumber{}
\issuenumber{}
\publicationyear{}
\papernumber{}
\startpage{}
\endpage{}
\doi{}
\MR{}
\Zbl{}
\received{}
\revised{}
\accepted{}
\published{}
\publishedonline{}
\proposed{}
\seconded{}
\corresponding{}
\editor{}
\version{}
\newtheorem{thm}{Theorem}[section]
\newtheorem{lem}[thm]{Lemma}
\newtheorem{prop}[thm]{Proposition}
\newtheorem{corl}[thm]{Corollary}

\theoremstyle{definition}
\newtheorem*{rem}{Remark}
\newtheorem*{expl}{Example}
\newtheorem*{defin}{Definition}

\numberwithin{equation}{section}

\makeatletter
\renewcommand{\thefigure}{\ifnum \c@section>\z@ \thesection.\fi
 \@arabic\c@figure}
\@addtoreset{figure}{section}
\makeatother

\begin{document}

\def\BB{\mathcal B} \def\Div{{\rm Div}} \def\SS{\mathcal S}
\def\Z{\mathbb Z} \def\clg{{\rm clg}} \def\ninf{{\rm ninf}}
\def\N{\mathbb N} \def\bh{{\rm bh}} \def\dpt{{\rm dpt}}
\def\com{{\rm com}} \def\rev{{\rm rev}}


\newcommand\blfootnote[1]{%
  \begingroup
  \renewcommand\thefootnote{}\footnote{#1}%
  \addtocounter{footnote}{-1}%
  \endgroup
}

\blfootnote{$^*$ Supported by CONICYT Beca Doctorado "Becas Chile" 72130288.}

\begin{abstract}
We introduce a structure on a Garside group that we call Dehornoy structure and we show that an iteration of such a structure leads to a left-order on the group. 
We define two conditions on a Garside group $G$ and we show that, if $G$ satisfies these two conditions, then $G$ has a Dehornoy structure. 
Then we show that the Artin groups of type $A$ and of type $I_2(m)$, $m \ge 4$, satisfy these conditions, and therefore have Dehornoy structures. 
As indicated by the terminology, one of the orders obtained by this method on the Artin groups of type $A$ coincides with the Dehornoy order.
\end{abstract}

\maketitle


\section{Introduction}

A group $G$ is said to be \emph{left-orderable} if there exists a total order $<$ on $G$ invariant by left-multiplication.
Recall that a subset $P$ of $G$ is a \emph{subsemigroup} if $\alpha \beta \in P$ for all $\alpha, \beta \in P$.
It is easily checked that a left-order $<$ on $G$ is determined by a subsemigroup $P$ such that $G = P \sqcup P^{-1} \sqcup \{ 1 \}$: we have $\alpha < \beta$ if and only if $\alpha^{-1} \beta \in P$.
In this case the subsemigroup $P$ is called the \emph{positive cone} of $<$.

The first explicit left-order on the braid group $\BB_n$ was determined by Dehornoy \cite{Dehor1}.
The fact that $\BB_n$ is left-orderable is important, but, furthermore, the Dehornoy order is interesting by itself, and there is a extensive literature on it.
We refer to Dehornoy--Dynnikov--Rolfsen--Wiest \cite{DDRW1} for a complete report on left-orders on braid groups and on the Dehornoy order in particular. 
The definition of the Dehornoy order is based on the following construction.

Let $G$ be a group and let $S = \{s_1, s_2, \dots, s_n \}$ be a finite ordered generating set for $G$.
Let $i \in \{1, 2, \dots, n\}$.
We say that $\alpha \in G$ is \emph{$s_i$-positive} (resp. \emph{$s_i$-negative}) if $\alpha$ is written in the form $\alpha = \alpha_0 s_i \alpha_1 \cdots s_i \alpha_m$ (resp. $\alpha = \alpha_0 s_i^{-1} \alpha_1 \cdots s_i^{-1} \alpha_m$) with $m \ge 1$ and $\alpha_0, \alpha_1, \dots, \alpha_m \in \langle s_{i+1}, \dots, s_n \rangle$.
For each $i \in \{1, 2, \dots, n\}$ we denote by $P_i^+$ (resp. $P_i^-$) the set of $s_i$-positive elements (resp. $s_i$-negative elements) of $G$.
The key point in the definition of the Dehornoy order is the following.

\begin{thm}[Dehornoy \cite{Dehor1}]\label{thm1_1}
Let $G = \BB_{n+1}$ be the braid group on $n+1$ strands and let $S = \{s_1, s_2, \dots, s_n \}$ be its standard generating set.
For each $i \in \{1,2, \dots, n \}$ we have the disjoint union $\langle s_i, s_{i+1}, \dots, s_n \rangle = P_i^+ \sqcup P_i^{-} \sqcup \langle s_{i+1}, \dots, s_n \rangle$.
\end{thm}

Let $G = \BB_{n+1}$ be the braid group on $n+1$ strands.
Set $P_D = P_1^+ \sqcup P_2^+ \sqcup \cdots \sqcup P_n^+$.
Then, by Theorem \ref{thm1_1}, $P_D$ is the positive cone for a left-order $<_D$ on $G$.
This is the \emph{Dehornoy order}.

A careful reader will notice that Theorem \ref{thm1_1} leads to more than one left-order on $\BB_{n+1}$.
Indeed, if $\epsilon = (\epsilon_1, \epsilon_2, \dots, \epsilon_n) \in \{ +,- \}^n$, then $P^\epsilon = P_1^{\epsilon_1} \sqcup P_2^{\epsilon_2} \sqcup \cdots \sqcup P_n^{\epsilon_n}$ is a positive cone for a left-order on $\BB_{n+1}$.
The case $\epsilon = (+,-,+, \dots)$ is particularly interesting because, by Dubrovina--Dubrovin \cite{DubDub1}, in this case $P^\epsilon$ determines an isolated left-order in the space of left-orders on $\BB_{n+1}$.

Our goal in the present paper is to extend the Dehornoy order to some Garside groups.

A first approach would consist on keeping the same definition, as follows.
Let $G$ be a group and let $S = \{s_1, s_2, \dots, s_n\}$ be a finite ordered generating set for $G$.
Again, we denote by $P_i^+$ (resp. $P_i^-$) the set of $s_i$-positive elements (resp. $s_i$-negative elements) of $G$.
Then we say that $S$ determines a \emph{Dehornoy structure} (\emph{in Ito's sense}) if, for each $i \in \{1, \dots, n \}$, we have the disjoint union $\langle s_i, s_{i+1}, \dots, s_n \rangle = P_i^+ \sqcup P_i^- \sqcup \langle s_{i+1}, \dots, s_n \rangle$.
In this case, as for the braid group, for each $\epsilon \in \{ +, - \}^n$ the set $P^\epsilon = P_1^{\epsilon_1} \sqcup P_2^{\epsilon_2} \sqcup \cdots \sqcup P_n^{\epsilon_n}$ is the positive cone for a left-order on $G$.
This approach was used by Ito \cite{Ito1} to construct isolated left-orders in the space of left-orders of some groups.

In the present paper we will consider another approach of the Dehornoy order in terms of Garside groups (see Dehornoy \cite{Dehor2}, Fromentin \cite{Fromen1}, Fromentin--Paris \cite{FroPar1}), and our definition of Dehornoy structure will be different from that in Ito's sense given above.

In Section \ref{Sec2} we recall some basic and preliminary definitions and results on Garside groups. 
We refer to Dehornoy et al. \cite{DehEtAl} for a full account on the theory.
In Section \ref{Sec3} we give our (new) definition of Dehornoy structure and show how such a structure leads to a left-order on the group (see Proposition \ref{prop3_1}).
Then we define two conditions on a Garside group, that we call Condition A and Condition B, and show that a Garside group which satisfies these two conditions has a Dehornoy structure (see Theorem \ref{thm3_2}).

The aim of the rest of the paper is to apply Theorem \ref{thm3_2} to the Artin groups of type A, that is, the braid groups, and the Artin groups of dihedral type. 
In Section \ref{Sec4} we prove that a braid group with its standard Garside structure satisfies Condition A and Condition B (see Theorem \ref{thm4_1}), and therefore has a Dehornoy structure in the sense of the definition of Section \ref{Sec3} (see Corollary \ref{corl4_2}). 
We also prove that the left-orders on the group induced by this structure are the same as the left-orders induced by Theorem \ref{thm1_1} (see Proposition \ref{prop4_4}), as expected.
Section \ref{Sec5} and  Section \ref{Sec6} are dedicated to the Artin groups of dihedral type. 
There is a difference between the even case, treated in Section \ref{Sec5}, and the odd case, treated in Section \ref{Sec6}.
The latter case requires much more calculations.
In both cases we show that such a group satisfies Condition A and Condition B, and therefore admits a Dehornoy structure. 
Then we show that the left-orders obtained from this Dehornoy structure can also be obtained via an embedding of the group in a braid group defined by Crisp \cite{Crisp1}. 


\section{Preliminaries}\label{Sec2}

Let $G$ be a group and let $M$ be a submonoid of $G$ such that $M \cap M^{-1}= \{ 1\}$.
Then we have two partial orders $\le_R$ and $\le_L$ on $G$ defined by $\alpha \le_R \beta$ if  $\beta \alpha^{-1} \in M$, and   $\alpha \le_L \beta$ if $\alpha^{-1} \beta \in M$.
For each $a \in M$ we set $\Div_R (a) = \{ b \in M \mid b  \le_R a\}$ and $\Div_L (a) = \{ b \in M \mid b  \le_L a\}$.
We say that $a \in M$ is \emph{balanced} if $\Div_R (a) = \Div_L (a)$.
In that case we set $\Div (a) = \Div_R (a) = \Div_L (a)$.   
We say that $M$ is \emph{Noetherian} if for each element $a \in M$ there is an integer $n \ge 1$ such that $a$ cannot be written as a product of more than $n$ non-trivial elements. 

\begin{defin}
Let $G$ be a group, let $M$ be a submonoid of $G$ such that $M \cap M^{-1}= \{1\}$, and let $\Delta$ be a balanced element of $M$.
We say that $G$ is a \emph{Garside group} with \emph{Garside structure} $(G, M, \Delta)$ if:
\begin{itemize}
\item[(a)]
$M$ is Noetherian; 
\item[(b)]
$\Div(\Delta)$ is finite, it generates $M$ as a monoid, and it generates $G$ as a group; 
\item[(c)]
$(G, \le_R)$ is a lattice.
\end{itemize}
\end{defin}

Let $(G, M, \Delta)$ be a Garside structure on $G$.
Then $\Delta$ is called the \emph{Garside element} and the elements of $\Div (\Delta)$ are called the \emph{simple elements} (of the Garside structure).
The lattice operations of $(G, \le_R)$ are denoted by $\wedge_R$ and $\vee_R$.
The ordered set $(G, \le_L)$ is also a lattice and its lattice operations are denoted by $\wedge_L$ and $\vee_L$.

Now take a Garside group $G$ with Garside structure $(G, M, \Delta)$ and set $\SS = \Div (\Delta) \setminus \{1\}$.
The word length of an element $\alpha \in G$ with respect to $\SS$ is denoted by $\lg (\alpha) = \lg_\SS (\alpha)$.
The right \emph{greedy normal form} of an element $a \in M$ is the unique expression $a = u_p \cdots u_2 u_1$ of $a$ over $\SS$ satisfying $(u_p \cdots u_i) \wedge_R \Delta = u_i$ for all $i \in \{1, \dots, p\}$.
We define the left \emph{greedy normal form} of an element of $M$ in a similar way.
The following two theorems contain several key results of the theory of Garside groups. 

\begin{thm}[Dehornoy--Paris \cite{DehPar1}, Dehornoy \cite{Dehor3}]\label{thm2_1}
\begin{itemize}
\item[(1)]
Let $a \in M$ and let $a = u_p \cdots u_2 u_1$ be the greedy normal form of $a$.
Then $\lg (a) = p$.
\item[(2)]
Let $\alpha \in G$.
There exists a unique pair $(a,b) \in M \times M$ such that $\alpha = a b^{-1}$ and $a \wedge_R b = 1$.
In that case we have $\lg (\alpha) = \lg (a) + \lg (b)$.
\end{itemize}
\end{thm}

The expression of $\alpha$ given in Theorem \ref{thm2_1}\,(2) is called the (right) \emph{orthogonal form} of $\alpha$.
The  \emph{left orthogonal form} of an element of $G$ is defined in a similar way.

We say that an element $a \in M$ is \emph{unmovable} if $\Delta \not \le_R a$ or, equivalently, if $\Delta \not \le_L a$.

\begin{thm}[Dehornoy--Paris \cite{DehPar1}, Dehornoy \cite{Dehor3}]\label{thm2_2}
Let $\alpha \in G$.
There exists a unique pair $(a,k) \in M \times \Z$ such that $a$ is unmovable and $\alpha = a \Delta^k$.
\end{thm}

The expression of $\alpha$ given above is called the (right) \emph{$\Delta$-form} of $\alpha$.
We define the \emph{left $\Delta$-form} of an element of $G$ in a similar way.

\begin{defin}
Let $\delta$ be a balanced element of $M$.
Denote by $G_\delta$ (resp. $M_\delta$) the subgroup of $G$ (resp. the submonoid of $M$) generated by $\Div (\delta)$.
We say that $(G_\delta, M_\delta, \delta)$ is a \emph{parabolic substructure} of $(G,M,\Delta)$ if $\delta$ is balanced and $\Div (\delta) = \Div (\Delta) \cap M_\delta$.
In that case $G_\delta$ is called a \emph{parabolic subgroup} of $G$ and $M_\delta$ is called a \emph{parabolic submonoid} of $M$. 
\end{defin}

\begin{rem}
Let $H$ be a parabolic subgroup of $G$.
Then there exists a unique parabolic substructure $(G_\delta, M_\delta, \delta)$ of $(G,M, \Delta)$ such that $H = G_\delta$.
Indeed, the above element $\delta$ should be the greatest element in $H \cap \Div (\Delta)$ for the order relation $\le_R$, hence $\delta$ is entirely determined by $H$.
Similarly, if $N$ is a parabolic submonoid of $M$, then there exists a unique parabolic substructure $(G_\delta, M_\delta, \delta)$ such that $N = M_\delta$, where $\delta$ is the greatest element of $\Div (\Delta) \cap N$ for the order relation $\le_R$.
So, we can speak of a parabolic subgroup or of a parabolic submonoid without necessarily specifying the corresponding element $\delta$ or the triple $(G_\delta, M_\delta, \delta)$.
\end{rem}

\begin{thm}[Godelle \cite{Godel1}]\label{thm2_3}
Let $(H,N, \delta)$ be a parabolic substructure of $(G, M, \Delta)$.
\begin{itemize}
\item[(1)]
$H$ is a Garside group with Garside structure $(H, N, \delta)$.
\item[(2)]
Let $a \in N$ and let $a = u_p \cdots u_2 u_1$ be the greedy normal form of $a$ with respect to $(G, M, \Delta)$.
Then $u_i \in \Div (\delta)$ for all $i \in \{1,2, \dots, p\}$ and $a = u_p \cdots u_2 u_1$ is the greedy normal form of $a$ with respect to $(H, N, \delta)$.
\item[(3)]
Let $\alpha, \beta \in H$ and $\gamma \in G$ such that $\alpha \le_R \gamma \le_R \beta$.
Then $\gamma \in H$.
\item[(4)]
Let $\alpha, \beta \in H$.
Then $\alpha \wedge_R \beta, \alpha \vee_R \beta \in H$.
\item[(5)]
Let $\alpha \in H$ and let $\alpha = a b^{-1}$ be the orthogonal form of $\alpha$ with respect to $(G, M, \Delta)$.
Then $a,b \in N$ and $\alpha = a b^{-1}$ is the orthogonal form of $\alpha$ with respect to $(H, N, \delta)$.
\end{itemize}
\end{thm}

\begin{expl}
Let $S$ be a finite set.
A \emph{Coxeter matrix} over $S$ is a square matrix $M = (m_{s,t})_{s,t \in S}$ indexed by the elements of $S$ with coefficients in $\N \cup \{ \infty\}$ such that $m_{s,s} = 1$ for all $s \in S$ and $m_{s,t} = m_{t,s} \ge 2$ for all $s,t \in S$, $s \neq t$.
If $s,t$ are two letters and $m$ is an integer $\ge 2$ we denote by $\Pi (s,t,m)$ the word $sts \cdots$ of length $m$.
In other words $\Pi (s,t,m) = (st)^{\frac{m}{2}}$ if $m$ is even and $\Pi (s, t, m) = (st)^{\frac{m-1}{2}}s$ if $m$ is odd. 
The \emph{Artin group} associated with $M$ is the group $A = A_M$ defined by the presentation
\[
A = \langle S \mid \Pi (s, t, m_{s,t}) = \Pi (t, s, m_{s,t}) \text{ for } s,t \in S,\ s \neq t \text{ and } m_{s,t} \neq \infty \rangle\,.
\]
The \emph{Coxeter group} associated with $M$ is the quotient $W = W_M$ of $A$ by the relations $s^2 =1$, $s \in S$.
We say that $A$ is of \emph{spherical type} if $W$ is finite.
The braid groups are the star examples of Artin groups of spherical type.

We denote by $A^+$ the monoid having the following monoid presentation. 
\[
A^+ = \langle S \mid \Pi (s, t, m_{s,t}) = \Pi (t, s, m_{s,t}) \text{ for } s,t \in S,\ s \neq t \text{ and } m_{s,t} \neq \infty \rangle^+\,.
\]
By Paris \cite{Paris1} the natural homomorphism $A^+ \to A$ is injective.
So, we can consider $A^+$ as a submonoid of $A$.
It is easily checked that $A^+ \cap (A^+)^{-1} = \{1\}$, hence we can consider the order relations $\le_R$ and $\le_L$ on $A$.
Suppose that $A$ is of spherical type. 
Then, by Brieskorn--Saito \cite{BriSai1} and Deligne \cite{Delig1}, for all $\alpha, \beta \in A$ the elements $\alpha \wedge_R \beta$ and $\alpha \vee_R \beta$ exist, and $(A, A^+, \Delta)$ is a Garside structure, where $\Delta = \vee_R S$.
Let $X$ be a subset of $S$ and let $A_X$ be the subgroup of $A$ generated by $X$.
Then, again by Brieskorn--Saito \cite{BriSai1} and Deligne \cite{Delig1}, $A_X$ is a parabolic subgroup of $A$ and it is an Artin group of spherical type. 
\end{expl}

The triple $(G,M,\Delta)$ denotes again an arbitrary Garside structure on a group $G$.
Besides the greedy normal forms, we will use some other normal forms of the elements of $M$ defined from a pair $(N_2, N_1)$ of parabolic submonoids of $M$.
Their definition is based on the following.

\begin{prop}[Dehornoy \cite{Dehor2}]\label{prop2_4}
Let $N$ be a parabolic submonoid of $M$.
For each $a \in M$ there exists a unique $b \in N$ such that $\{ c \in N \mid c \le_R a \} = \{ c \in N \mid c \le_R b \}$.
\end{prop}

The element $b$ of Proposition \ref{prop2_4} is called the (right) \emph{$N$-tail} of $a$ and is denoted by $b = \tau_N(a) = \tau_{N,R} (a)$.
We define in a similar way the left \emph{$N$-tail} of $a$, denoted by $\tau_{N,L} (a)$.

Now, assume that $N_1$ and $N_2$ are two parabolic submonoids of $M$ such that $N_2 \cup N_1$ generates $M$.
Then each nontrivial element $a \in M$ is uniquely written in the form $a= a_p \cdots a_2 a_1$ where $a_p \neq 1$, $a_i=\tau_{N_1}(a_p \cdots a_i)$ if $i$ is odd, and $a_i = \tau_{N_2} (a_p \cdots a_i)$ if $i$ is even.
This expression is called the (right) \emph{alternating form} of $a$ with respect to $(N_2, N_1)$.
Note that we may have $a_1=1$, but $a_i \neq 1$ for all $i \in \{2, \dots, p \}$.
The number $p$ is called the \emph{$(N_2, N_1)$-breadth} of $a$ and is denoted by $p = \bh(a) = \bh_{N_2,N_1} (a)$.
By extension we set $\bh (1) = 1$ so that $a \in N_1 \Leftrightarrow \bh(a)=1$.

Now, consider the standard Garside structure $(\BB_{n+1}, \BB_{n+1}^+, \Delta)$ on the braid group $\BB_{n+1}$.
Let $S = \{s_1, s_2, \dots, s_n\}$ be the standard generating system of $\BB_{n+1}$, $N_1$ be the submonoid of $\BB_{n+1}^+$ generated by $\{s_2, \dots, s_n\}$, and $N_2$ be the submonoid generated by $\{s_1, \dots, s_{n-1} \}$.
Then $N_1$ and $N_2$ are parabolic submonoids of $\BB_{n+1}^+$ and they are both isomorphic to $\BB_n^+$.
Observe that $N_1 \cup N_2$ generates $\BB_{n+1}^+$, hence we can consider alternating forms with respect to $(N_2, N_1)$.
The definitions of the next section are inspired by the following.

\begin{thm}[Fromentin--Paris \cite{FroPar1}]\label{thm2_5}
Let $a \in \BB_{n+1}^+$ and $k \in \Z$.
Then $\Delta^{-k} a$ is $s_1$-negative if and only if $k \ge \max \{1, \bh (a) -1 \}$.
\end{thm}


\section{Orders on Garside groups}\label{Sec3}

We consider a Garside structure $(G, M, \Delta)$ on a Garside group $G$ and two parabolic substructures $(H, N, \Lambda)$ and $(G_1, M_1, \Delta_1)$.
We assume that $N \neq M$, $M_1 \neq M$, $N \cup M_1$ generates $M$, $\Delta$ is central in $G$, and $\Delta_1$ is central in $G_1$.
Note that the assumption ``$\Delta$ is central in $G$'' is not so restrictive since, by Dehornoy \cite{Dehor3}, if $(G,M, \Delta)$ is a Garside structure, then $(G, M, \Delta^k)$ is also a Garside structure for each $k \ge 1$, and there exists $k \ge 1$ such that $\Delta^k$ is central in $G$.
We will consider alternating forms with respect to $(N,M_1)$.

The \emph{depth} of an element $a \in M$, denoted by $\dpt (a)$, is $\dpt (a) = \frac{ \bh (a) -1}{2}$ if $\bh (a)$ is odd and is $\dpt (a) = \frac{\bh (a)}{2}$ if $\bh (a)$ is even. 
In other words, if $a = a_p \cdots a_2 a_1$ is the alternating form of $a$, then $\dpt (a)$ is the number of indices $i \in \{1, \dots, p\}$ such that $a_i\not \in M_1$ (that is, the number of even indices). 
Note that $a \in M_1$ if and only if $\dpt(a)=0$.

\begin{defin}
Let $\alpha \in G$ and let $\alpha = a \Delta^{-k}$ be its $\Delta$-form.
We say that $\alpha$ is \emph{$(H,G_1)$-negative} if $k \ge 1$ and $\dpt (a) < \dpt (\Delta^k)$.
We say that $\alpha$ is \emph{$(H, G_1)$-positive} if $\alpha^{-1}$ is $(H, G_1)$-negative. 
We denote by $P = P_{H,G_1}$ the set $(H,G_1)$-positive elements and by $P^{-1}$ the set of $(H, G_1)$-negative elements.
\end{defin}

\begin{defin}
We say that $(H,G_1)$ is a \emph{Dehornoy structure} if $P$ satisfies the following conditions:
\begin{itemize}
\item[(a)]
$P P \subset P$,
\item[(b)]
$G_1 P G_1 \subset P$,
\item[(c)]
we have the disjoint union $G = P \sqcup P^{-1} \sqcup G_1$.
\end{itemize}
\end{defin}

Our goal in this section is to prove a criterion for $(H,G_1)$ to be a Dehornoy structure.
But, before, we show how the orders appear in this context.

Suppose given two sequences of parabolic subgroups $G_0 = G, G_1, \dots, G_n$ and $H_1, \dots,
\allowbreak
H_n$ such that $G_{i+1}, H_{i+1} \subset G_{i}$ and $(H_{i+1},G_{i+1})$ is a Dehornoy structure on $G_{i}$ for all $i \in \{0,1, \dots, n-1\}$ and $G_n \simeq \Z$.
For each  $i \in \{0,1, \dots, n-1\}$ we denote by $P_i$ the set of $(H_{i+1}, G_{i+1})$-positive elements of $G_{i}$.
On the other hand, we choose a generator $\alpha_n$ of $G_n$ and we set $P_{n} = \{ \alpha_n^k \mid k \ge 1 \}$.
For each $\epsilon = (\epsilon_0, \epsilon_1, \dots, \epsilon_n) \in \{ \pm 1 \}^{n+1}$ we set $P^\epsilon = P_0^{\epsilon_0} \sqcup P_1^{\epsilon_1} \sqcup \cdots \sqcup P_n^{\epsilon_n}$.

\begin{prop}\label{prop3_1}
Under the above assumptions $P^\epsilon$ is the positive cone for a left-order on $G$.
\end{prop}

\begin{proof}
We must prove that we have a disjoint union $G = P^\epsilon \sqcup (P^\epsilon)^{-1} \sqcup \{1 \}$ and that $P^\epsilon P^\epsilon \subset P^\epsilon$.
The fact that we have a disjoint union $G = P^\epsilon \sqcup (P^\epsilon)^{-1} \sqcup \{1\}$ follows directly from Condition (c)  of the definition. 
Let $\alpha, \beta \in P^\epsilon$.
Let $i,j \in \{0,1, \dots, n\}$ such that $\alpha \in P_i^{\epsilon_i}$ and $\beta \in P_j^{\epsilon_j}$.
If $i <j$, then, by Condition (b) of the definition, $\alpha \beta \in P_i^{\epsilon_i} \subset P^\epsilon$.
Similarly, if $i > j$, then $\alpha \beta \in P_j^{\epsilon_j} \subset P^\epsilon$.
If $i=j$, then, by Condition (a) of the definition, $\alpha \beta \in P_i^{\epsilon_i} \subset P^\epsilon$.
\end{proof}

\begin{defin}
Let $\zeta \ge 1$ be an integer.
We say that the pair $(H,G_1)$ satisfies \emph{Condition A with constant $\zeta$} if $\dpt (\Delta^k) = \zeta k+1$ for all $k \ge 1$.
\end{defin}

We set $\theta = \Delta \Delta_1^{-1} = \Delta_1^{-1} \Delta \in M$.
We say that an element $a \in M$ is a \emph{theta element} if it is of the form $a = \theta^k a_0$ with $k \ge 1$ and $a_0 \in M_1$.
We denote by $\Theta$ the set of theta elements of $M$ and we set $\bar \Theta = \Theta \cup M_1$.

\begin{defin}
Let $\zeta \ge 1$ be an integer. 
Let $(a,b) \in (M \times M) \setminus (\bar \Theta \times \bar \Theta)$ such that $a,b$ are both unmovable.
Let $ab = c \Delta^t$ be the $\Delta$-form of $ab$.
We say that $(a,b)$ satisfies \emph{Condition B with constant $\zeta$} if there exists $\varepsilon \in \{ 0, 1\}$ such that
\begin{itemize}
\item[(a)]
$\dpt(c) = \dpt(a) + \dpt(b) - \zeta t - \varepsilon$,
\item[(b)]
$\varepsilon=1$ if either $a \in \Theta$, or $b \in \Theta$, or $c \in M_1$.
\end{itemize}
We say that $(H,G_1)$ satisfies \emph{Condition B with constant $\zeta$} if each pair $(a,b) \in (M \times M) \setminus (\bar \Theta \times \bar \Theta)$ as above satisfies Condition B with constant $\zeta$.
\end{defin}

\begin{thm}\label{thm3_2}
If there exists a constant $\zeta \ge 1$ such that $(H,G_1)$ satisfies Condition A with constant $\zeta$ and Condition B with constant $\zeta$, then $(H,G_1)$ is a Dehornoy structure.
\end{thm}

Let $\zeta \ge 1$ be an integer.
From here until the end of the section we assume that $(H,G_1)$ satisfies Condition A with constant $\zeta$ and Condition B with constant $\zeta$. 
Our goal is then to prove that $(H,G_1)$ is a Dehornoy structure, that is, to prove Theorem \ref{thm3_2}.

Let $a$ be an unmovable element of $M$ and let $p=\lg (a)$.
Then $p$ is the smallest integer $\ge 0$ such that $a \le_R \Delta^p$.
Let $\com (a) \in M$ such that $a\, \com(a) = \Delta^p$.
Then, by El-Rifai--Morton \cite{ElRMor1}, $\com (a)$ is unmovable,  $\lg(\com(a)) = p$, and $a^{-1} = \com (a) \Delta^{-p}$ is the $\Delta$-form of $a^{-1}$.
Note that $a\, \com(a) = \com(a)\, a = \Delta^p$ since $\Delta$ is central. 
In particular, $\com (\com(a)) = a$.

\begin{lem}\label{lem3_3}
\begin{itemize}
\item[(1)]
Let $a \in M_1$.
Then $\theta \wedge_R a = 1$ and $\theta \vee_R a = \theta a = a \theta$.
\item[(2)]
Let $a = \theta^k a_0$ be a theta element, where $k \ge 1$ and $a_0 \in M_1$.
Then $\dpt(a) = \zeta k+1$.
\item[(3)]
Let $a = \theta^k a_0$ be a theta element, where $k \ge 1$ and $a_0 \in M_1$.
Then $a$ is unmovable if and only if $a_0$ is unmovable in $M_1$ (that is, if and only if $\Delta_1 \not\le_R a_0$).
\item[(4)]
Let $a$ be an unmovable element of $M$.
We have $a \in \bar \Theta$ if and only if $\com(a) \in \bar \Theta$.
\item[(5)]
Let $\alpha \in G_1 \setminus M_1$.
Then $\alpha$ has a $\Delta$-form of the form $\alpha = a \Delta^{-k}$ where $k \ge 1$ and $a=\theta^k a_0 \in \Theta$ with $a_0 \in M_1$.
\item[(6)]
Let $a \in \bar \Theta$ and $b \in M \setminus \bar \Theta$.
Then $ab \in M \setminus \bar \Theta$ and $ba \in M \setminus \bar \Theta$.
\end{itemize}
\end{lem}

\begin{proof}
{\it Part (1):}
Let $a \in M_1$.
Let $u = a \wedge_R \theta$.
We have $u \le_R \theta$, hence $u \Delta_1 \le_R \theta \Delta_1 = \Delta$, and therefore $u \Delta_1 \in \Div (\Delta)$.
On the other hand, since $u \le_R a$, we have $u \in M_1$, hence $u \Delta_1 \in M_1$.
So, $u \Delta_1 \in \Div (\Delta) \cap M_1 = \Div (\Delta_1)$, thus $u=1$.
Let $v = a \vee_R \theta$.
Since $\Delta$ and $\Delta_1$ commute with $a$, we have $\theta a = a \theta$.
In particular, $v \le_R a \theta$.
Let $x_1 \in M$ such that $v = x_1 \theta$.
Then $x_1 \le_R a$ and, since $M_1$ is a parabolic submonoid, $x_1 \in M_1$ and there exists $x_2 \in M_1$ such that $x_2 x_1 = a$.
So, $a = x_2 x_1 \le_R v = x_1 \theta = \theta x_1$, hence $x_2 \le_R \theta$, and therefore, since $a \wedge_R \theta =1$, we have $x_2=1$.
Thus $x_1=a$ and $v = a \theta = \theta a$.

{\it Part (2):}
It is clear that $\dpt(a) = \dpt(aa_0)$ for all $a \in M$ and all $a_0 \in M_1$.
Let $a = \theta^k a_0$ be a theta element.
Then $\dpt(a) = \dpt( \theta^k) = \dpt(\theta^k \Delta_1^k) = \dpt (\Delta^k) = \zeta k+1$.

{\it Part (3):}
Let $a = \theta^k a_0$ be a theta element.
Suppose that $\Delta_1 \le_R a_0$.
Let $a_1 \in M_1$ such that $a_0 = a_1 \Delta_1$.
Then $a = \theta^k a_1 \Delta_1 = \theta^{k-1} a_1 \theta \Delta_1 = \theta^{k-1} a_1 \Delta$, hence $\Delta \le_R a$.
Now suppose that $\Delta \le_R a$.
By Part (1) we have $\tau_{M_1} (a) = a_0$.
Since $\Delta \le _R a$, we have $\Delta_1 \le_R a$, hence $\Delta_1 \le_R \tau_{M_1} (a) =a_0$.

{\it Part (4):}
Let $a$ be an unmovable element of $M$ and let $p = \lg (a)$.
Suppose that $a \in M_1$.
Let $b \in M_1$ such that $ab = \Delta_1^p$.
Then $a \theta^p b = \theta^p a b = \theta^p \Delta_1^p = \Delta^p$, hence $\com (a) = \theta^p b \in \bar \Theta$.
Suppose that $a = \theta^k a_0$ where $k \ge 1$ and $a_0 \in M_1$.
We have $a = \theta^k a_0 \le_R \Delta^p = \theta^p \Delta_1^p$ hence, by Part (1), $a_0 \le_R \Delta_1^p$ and $k \le p$.
Let $b_0 \in M_1$ such that $a_0 b_0 = \Delta_1^p$.
Then $a \theta^{p-k} b_0 = \theta^k a_0 \theta^{p-k} b_0 = \theta^p a_0 b_0 = \theta^p \Delta_1^p = \Delta^p$, hence $\com (a) = \theta^{p-k} b_0 \in \bar \Theta$.
So, if $a \in \bar \Theta$, then $\com (a) \in \bar \Theta$.
Now, since $\com (\com (a)) = a$ for each unmovable element $a$ of $M$, we have $a \in \bar \Theta$ if and only if $\com (a) \in \bar \Theta$.

{\it Part (5):}
Let $\alpha \in G_1 \setminus M_1$.
Since $\alpha \not\in M_1$ the $\Delta_1$-form of $\alpha$ is of the form $\alpha = a \Delta_1^{-k}$ with $a \in M_1$, $\Delta_1 \not\le_R a$ and $k \ge 1$.
Then $\alpha = a (\theta\Delta^{-1})^k = \theta^k a \Delta^{-k}$ and $\theta^k a$ is unmovable by Part (3) of the lemma. 

{\it Part (6):}
Take $a,b \in M$.
We assume that $a, ab \in \bar \Theta$ and we turn to prove that $b \in \bar \Theta$.
We write $ab = \theta^t c$ where $t \ge 0$ and $c \in M_1$.
On the other hand we know by Part (4) that $\com (a) \in \bar \Theta$, hence $\com(a)$ is of the form $\com (a) = \theta^k a_0$ with $k \ge 0$ and $a_0 \in M_1$, and therefore $a^{-1}$ is of the form $a^{-1} = \theta^k a_0 \Delta^{-\ell} = \theta^{k-\ell} a_0 \Delta_1^{-\ell}$ where $\ell = \lg (a)$.
So, $b \Delta_1^\ell=  \theta^{t+k-\ell} a_0 c$.
If we had $t+k-\ell <0$, then we would have $\theta^{\ell -t -k} b \Delta_1^\ell = a_0 c \in M_1$, hence we would have $\theta^{\ell - t - k} \in M_1$, which contradicts Part (1).
So, $t+k-\ell \ge 0$.
By Part (1) we have $\tau_{M_1}(\theta^{t+k-\ell}a_0c) = a_0 c$, hence $\Delta_1 \le_R a_0 c$.
Let $b_0 \in M_1$ such that $b_0 \Delta_1^\ell = a_0 c$.
Then $b = \theta^{t+k -\ell} b_0 \in \bar \Theta$.
We show in the same way that, if $a,ba \in \bar \Theta$, then $b \in \bar \Theta$.
\end{proof}

\begin{lem}\label{lem3_4}
We have $P^{-1} P^{-1} \subset P^{-1}$.
\end{lem}

\begin{proof}
Let $\alpha, \beta \in P^{-1}$.
Let $\alpha = a \Delta^{-k}$ and $\beta = b \Delta^{-\ell}$ be the $\Delta$-forms of $\alpha$ and $\beta$, respectively.
Since $\alpha, \beta \in P^{-1}$, we have $k, \ell \ge 1$, $\dpt(a) \le \dpt(\Delta^k)-1 = \zeta k$ and $\dpt (b) \le \dpt(\Delta^\ell)-1 = \zeta \ell$.
Let $ab = c \Delta^t$ be the $\Delta$-form of $ab$.
Then the $\Delta$-form of $\alpha \beta$ is $\alpha \beta = c \Delta^{-k - \ell + t}$.
We must show that $\alpha \beta \in P^{-1}$, that is, $k + \ell -t \ge 1$ and $\dpt (c) \le \dpt (\Delta^{k+\ell-t})-1 = \zeta(k + \ell-t)$.

{\it Case 1: $a,b \in M_1$.}
Then $t=0$ and $c = ab$, hence $k+\ell -t = k+\ell \ge 1$ and $\dpt (c) = 0 \le \zeta (k+\ell) = \zeta(k + \ell -t)$.

{\it Case 2: $a \in M_1$ and $b \in \Theta$.}
We write $b = \theta^u b_0$ where $u \ge 1$ and $b_0 \in M_1$.
By Lemma \ref{lem3_3}\,(3) we have $\dpt(b) = \zeta u+1\le \zeta \ell$, hence $u < \ell$.
Let $a b_0 = c_0 \Delta_1^v$ be the $\Delta_1$-form of $a b_0$.
If $v < u$, then $t=v$ and $c = \theta^{u-v} c_0$, hence $k+\ell -t \ge \ell -v \ge \ell -u \ge 1$ and $\dpt (c) = \zeta(u-v)+1 = \zeta u - \zeta t +1 \le \zeta \ell - \zeta t \le \zeta (k+\ell -t)$.
If $v \ge u$, then $t=u$ and $c = \Delta_1^{v-u} c_0 \in M_1$, hence $k + \ell - t = k + \ell -u \ge \ell -u \ge 1$ and $\dpt (c) = 0 \le \zeta (k+\ell-t)$.
The case ``$a \in \Theta$ and $b \in M_1$'' can be proved in a similar way.

{\it Case 3: $a,b \in \Theta$.}
We set $a = \theta^u a_0$ and $b = \theta^v b_0$, where $u,v \ge 1$ and $a_0, b_0 \in M_1$.
Since $\dpt(a) = \zeta u+1 \le \zeta k$, we have $u < k$.
Similarly, we have $v < \ell$.
Let $a_0 b_0 = c_0 \Delta_1^w$ be the $\Delta_1$-form of $a_0 b_0$.
If $w < u+v$, then $t=w$ and $c = \theta^{u+v-w} c_0$, hence $k + \ell -t \ge k + \ell -(u+v) = (k-u) + (\ell -v) \ge 1$ and $\dpt (c) = \zeta (u+v-w)+1 = \zeta u+1 + \zeta v -\zeta t \le \zeta k + \zeta \ell - \zeta t = \zeta (k + \ell-t)$.
If $w \ge u+v$, then $t = u+v$ and $c = c_0 \Delta_1^{w-u-v} \in M_1$, hence $k + \ell -t = k + \ell - (u+v) = (k -u) + (\ell -v) \ge 1$ and $\dpt (c) = 0 \le \zeta (k + \ell-t)$.

{\it Case 4: either $a \not\in \bar \Theta$, or $b \not \in \bar \Theta$.}
Since $(H,G_1)$ satisfies Condition B with constant $\zeta$, there exists $\varepsilon \in \{ 0, 1 \}$ such that $\dpt (c) = \dpt(a) + \dpt(b) - \zeta t - \varepsilon$.
If $c \in M_1$, then $\varepsilon = 1$ and 
\[
0 = \dpt(c) = \dpt(a) + \dpt(b) - \zeta t - 1 \le \zeta k + \zeta \ell - \zeta t -1 < \zeta (k+\ell -t)\,.
\]
This (strict) inequality also implies that $k + \ell -t \ge 1$.
If $c \not \in M_1$, then 
\[
1 \le \dpt(c) \le \dpt(a) + \dpt(b) - \zeta t \le \zeta k + \zeta \ell - \zeta t = \zeta (k + \ell -t)\,.
\]
Again, this inequality also implies that $k + \ell -t \ge 1$.
\end{proof}

\begin{lem}\label{lem3_5}
We have $G_1 P^{-1} G_1 \subset P^{-1}$.
\end{lem}

\begin{proof}
We take $\alpha \in G_1$ and $\beta \in P^{-1}$ and we turn to prove that $\alpha \beta \in P^{-1}$.
The proof of the inclusion $\beta \alpha \in P^{-1}$ is made in a similar way.
Let $\alpha = a \Delta^{-k}$ and $\beta = b \Delta ^{-\ell}$ be the $\Delta$-forms of $\alpha$ and $\beta$, respectively.
Since $\beta \in P^{-1}$ we have $\ell \ge 1$ and $\dpt (b) \le \dpt (\Delta^\ell)-1 = \zeta \ell$.
Let $ab = c \Delta^t$ be the $\Delta$-form of $ab$.
Then the $\Delta$-form of $\alpha \beta$ is $\alpha \beta = c \Delta^{-k-\ell+t}$.
We must show that $k+\ell-t \ge 1$ and $\dpt(c) \le  \zeta (k + \ell -t)$.

{\it Case 1: $\alpha \in M_1$ and $b \in M_1$.}
We have $k=0$, $\alpha = a$, $t=0$ and $c=ab \in M_1$.
Thus $k+\ell-t = \ell \ge 1$ and $0 = \dpt (c) \le \zeta (k + \ell - t)$.

{\it Case 2: $\alpha \in M_1$ and $b \in \Theta$.}
We have $k=0$, $\alpha = a$ and $b = \theta^v b_0$ where $v \ge 1$ and $b_0 \in M_1$.
We also have $\dpt(b) = \zeta v+1 \le \zeta \ell$, hence $v < \ell$.
Let $ab_0 = c_0 \Delta_1^u$ be the $\Delta_1$-form of $ab_0$.
If $u < v$, then $t=u$ and $c=\theta^{v-u} c_0$, hence $k + \ell -t =\ell -u \ge \ell -v \ge 1$ and $\dpt (c) = \zeta (v-u)+1 = \zeta v+1 - \zeta t \le \zeta \ell - \zeta t = \zeta (k + \ell -t)$.
If $u \ge v$, then $t = v$ and $c = \Delta_1^{u-v} c_0 \in M_1$, hence $k + \ell -t = \ell -v \ge 1$ and $0 = \dpt (c) \le \zeta (k + \ell -t)$.

{\it Case 3: $\alpha \in M_1$ and $b \in M \setminus \bar \Theta$.}
We have $k=0$ and $\alpha = a$.
On the other hand, by Lemma \ref{lem3_3}\,(6), we have $ab \in M\setminus \bar \Theta$, hence $c \not\in M_1$, and therefore $\dpt (c) \ge 1$.
Since $(H, G_1)$ satisfies Condition B with constant $\zeta$, there exists $\varepsilon \in \{ 0, 1 \}$ such that $\dpt (c) = \dpt (a) + \dpt (b) - \zeta t - \varepsilon$.
So, 
\[
1 \le \dpt (c) \le 0 + \zeta \ell - \zeta t = \zeta (k + \ell -t)\,.
\]
This inequality also implies that $k + \ell -t \ge 1$.

{\it Case 4: $\alpha \not \in M_1$ and $b \in M_1$.}
By Lemma \ref{lem3_3}\,(5) we have $k \ge 1$ and $a = \theta^k a_0$ with $a_0 \in M_1$.
Let $a_0 b = c_0 \Delta_1^u$ be the $\Delta_1$-form of $a_0 b$.
If $u <k$, then $t=u$ and $c = \theta^{k-u} c_0$, hence $k + \ell -t \ge \ell \ge 1$ and $\dpt (c) = \zeta (k-u)+1 \le \zeta k -\zeta t + \zeta \ell \le \zeta (k+\ell -t)$.
If $u \ge k$, then $t=k$ and $c = c_0 \Delta_1^{u-k} \in M_1$, hence $k + \ell -t = \ell \ge 1$ and $0 = \dpt (c) \le \zeta (k + \ell -t)$.

{\it Case 5: $\alpha \not \in M_1$ and $b \in \Theta$.}
By Lemma \ref{lem3_3}\,(5) we have $k \ge 1$ and $a = \theta^k a_0$ with $a_0 \in M_1$.
On the other hand, $b$ is written $b = \theta^v b_0$ with $v \ge 1$ and $b_0 \in M_1$.
Since $\dpt(b) = \zeta v+1 \le \zeta \ell$, we have $v < \ell$.
Let $a_0b_0 = c_0 \Delta_1^w$ be the $\Delta_1$-form of $a_0 b_0$.
If $w < k+v$, then $t = w$ and $c = \theta^{k+v-w}c_0$, hence $k + \ell -t \ge k + v -w \ge 1$ and $\dpt (c) = \zeta (k + v - w)+1 = \zeta k + \zeta v+1 - \zeta t \le \zeta k + \zeta \ell - \zeta t = \zeta (k + \ell -t)$.
If $w \ge k+v$, then $t = k +v$ and $c = c_0 \Delta_1^{w-k-v} \in M_1$, hence $k + \ell - t = \ell -v \ge 1$ and $0 = \dpt (c) \le \zeta (k + \ell -t)$.

{\it Case 6: $\alpha \not \in M_1$ and $b \in M \setminus \bar \Theta$.}
By Lemma \ref{lem3_3}\,(5) we have $k \ge 1$ and $a = \theta^k a_0$ with $a_0 \in M_1$.
On the other hand, by Lemma \ref{lem3_3}\,(6), we have $ab \in M\setminus \bar \Theta$, hence $c \not\in M_1$, and therefore $\dpt (c) \ge 1$.
Since $(H, G_1)$ satisfies Condition B with constant $\zeta$ and $a \in \Theta$, $\dpt (c) = \dpt (a) + \dpt (b) - \zeta t -1$.
So, 
\[
1 \le \dpt (c) \le \zeta k+1 + \zeta \ell - \zeta t -1 = \zeta (k + \ell -t)\,.
\]
This inequality also implies that $k + \ell -t \ge 1$.
\end{proof}

\begin{lem}\label{lem3_6}
We have $G_1 \cap (P \cup P^{-1}) = \emptyset$.
\end{lem}

\begin{proof}
Let $\alpha \in G_1$ and let $\alpha = a \Delta^{-k}$ be the $\Delta$-form of $\alpha$.
If $\alpha \in M_1$, then $k=0$ and $\alpha = a$, thus $\alpha \not\in P^{-1}$.
If $\alpha \not\in M_1$, then, by Lemma \ref{lem3_3}\,(5), we have $k \ge 1$ and $a = \theta^k a_0$ where $a_0 \in M_1$, hence $\dpt(a) = \zeta k+1  = \dpt (\Delta^k)$, and therefore $\alpha \not\in P^{-1}$.
Since $\alpha^{-1} \in G_1$, we also have $\alpha^{-1} \not\in P^{-1}$, hence $\alpha \not\in P$.
\end{proof}

\begin{lem}\label{lem3_7}
We have $P \cap P^{-1} = \emptyset$.
\end{lem}

\begin{proof}
Let $\alpha \in P^{-1}$ and let $\alpha = a \Delta^{-k}$ be its $\Delta$-form. 
By definition we have $k \ge 1$ and $\dpt (a) < \dpt (\Delta^k) = \zeta k+1$.
Let $\ell = \lg (a)$.
Then the $\Delta$-form of $\alpha^{-1}$ is $\alpha^{-1} = \com (a) \Delta^{k-\ell}$.
We are going to show that $\alpha^{-1} \not\in P^{-1}$, that is, either $k - \ell \ge 0$ or $\dpt (\com (a)) \ge \zeta (\ell-k)+1$.

{\it Case 1: $a \in M_1$.}
Let $b \in M_1$ such that $ab = \Delta_1^\ell$.
We have $a^{-1} = b \Delta_1^{-\ell} = b \theta^\ell \Delta^{-\ell} = \theta^\ell b \Delta^{-\ell}$, hence $\com (a) = \theta^\ell b$, and therefore, $\dpt (\com (a)) = \zeta \ell + 1 > \zeta (\ell -k)+1$, since $k \ge 1$.
So, $\alpha^{-1} \not\in P^{-1}$.

{\it Case 2: $a \in \Theta$.}
We write $a = \theta^u a_0$ where $a_0 \in M_1$ and $u \ge 1$.
We have $\dpt(a) = \zeta u+1 \le \zeta k$, hence $u < k$.
Let $t \ge 0$ be the length of $a_0$ and let $b_0 \in M_1$ such that $a_0 b_0 = \Delta_1^t$.
We have $a_0^{-1} = b_0 \Delta_1^{-t} =  \theta^t b_0 \Delta^{-t}$, hence $a^{-1} =  \theta^{t-u} b_0 \Delta^{-t}$, and therefore $\alpha^{-1} = \theta^{t-u} b_0 \Delta^{k-t}$.
If $u<t$, then $\com (a) = \theta^{t-u} b_0$ and $\dpt (\com (a)) = \zeta (t-u)+1 > \zeta (t-k) +1$, hence $\alpha^{-1} \not\in P^{-1}$.
If $u \ge t$, then $\alpha^{-1} =  \theta^{-u+t} b_0 \Delta^{k-t} = b_0 \Delta_1^{u-t} \Delta^{k-t-u+t} = b_0 \Delta_1^{u-t} \Delta^{k-u}$ and $k -u \ge 1$, hence $\alpha^{-1} \not\in P^{-1}$.

{\it Case 3: $a \in M \setminus \bar \Theta$.}
Recall that $a\, \com (a) = \Delta^\ell$.
Since $(H,G_1)$ satisfies Condition B with constant $\zeta$ and $1 \in M_1$, we have $0 = \dpt(1) = \dpt(a) + \dpt (\com (a)) - \zeta \ell -1$, hence 
\[
\dpt (\com (a)) = \zeta \ell + 1 - \dpt (a)  \ge \zeta \ell +1 - \zeta k = \zeta (\ell-k)+1\,,
\]
and therefore $\alpha^{-1} \not\in P^{-1}$.
\end{proof}

\begin{lem}\label{lem3_8}
We have $G = P \cup P^{-1} \cup G_1$.
\end{lem}

\begin{proof}
We take $\alpha \in G$ and we assume that $\alpha \not\in (P^{-1} \cup G_1)$.
We are going to show that $\alpha \in P$, that is, $\alpha^{-1} \in P^{-1}$.
Let $\alpha = a \Delta^k$ be the $\Delta$-form of $\alpha$ and let $\ell$ be the length of $a$.
Then the $\Delta$-form of $\alpha^{-1}$ is $\com (a) \Delta^{-k -\ell}$.

{\it Case 1: $a \in M_1$.}
Then $k \ge 1$ because $\alpha \not\in (P^{-1} \cup G_1)$.
If $a=1$, then $\alpha^{-1} = \Delta^{-k} \in P^{-1}$.
So, we can assume that $a \neq 1$, and therefore $\ell \ge 1$.
Let $b \in M_1$ such that $ab = \Delta_1^\ell$.
We have $a^{-1} = \theta^\ell b \Delta^{-\ell}$, hence $\alpha^{-1} = \theta^\ell b \Delta^{-k-\ell}$ and $\com (a) = \theta^\ell b$.
Then $k + \ell \ge 1$ and $\dpt (\com (a)) = \zeta \ell + 1 \le \zeta \ell + \zeta k = \zeta (\ell + k)$, hence $\alpha^{-1} \in P^{-1}$.

{\it Case 2: $a \in \Theta$.}
We write $a = \theta^u a_0$ where $u \ge 1$ and $a_0 \in M_1$.
Since $\alpha \not \in P^{-1}$ we have $\dpt(a) = \zeta u+1 \ge \zeta (-k)+1$, hence $u \ge -k$.
We also have $u \neq -k$, otherwise we would have $\alpha = a_0 \Delta_1^{-u} \in G_1$.
So, $u > -k$.
Let $t$ be the length of $a_0$ and let $b_0 \in M_1$ such that $a_0b_0 = \Delta_1^t$.
We have $a_0^{-1} = b_0 \Delta_1^{-t}$, hence $a^{-1} = \theta^{t-u} b_0 \Delta^{-t}$, and therefore $\alpha^{-1} = \theta^{t-u} b_0 \Delta^{-k-t}$.
If $u < t$, then $\com(a) = \theta^{t-u} b_0$, $k + t > k+u \ge 1$ and $\dpt (\com (a)) = \zeta (t-u)+1 < \zeta (t+k) +1 = \dpt (\Delta^{t+k})$, hence $\alpha^{-1} \in P^{-1}$.
If $u \ge t$, then $\alpha^{-1} = b_0 \Delta_1^{u-t} \Delta^{-k-u}$, $\com (a) = b_0 \Delta_1^{u-t} \in M_1$, $k+u \ge 1$, and $\dpt (\com (a)) = 0 \le \zeta (k+u)$, hence $\alpha^{-1} \in P^{-1}$.

{\it Case 3: $a \in M \setminus \bar \Theta$.}
Since $(H,G_1)$ satisfies Condition B with constant $\zeta$, we have $0 = \dpt (1) = \dpt (a) + \dpt (\com (a)) - \zeta \ell -1$.
On the other hand, since $\Delta^\ell \in \bar \Theta$, by Lemma \ref{lem3_3}\,(6), $\com (a) \not \in \bar \Theta$, hence $\com (a) \not \in M_1$, and therefore $\dpt (\com (a)) \ge 1$.
Moreover, since $\alpha \not \in P^{-1}$, we have $\dpt (a) \ge \zeta (-k)+1$.
So,
\[
1 \le \dpt (\com (a)) = \zeta \ell +1 - \dpt(a)  \le \zeta \ell + 1 + \zeta k -1 = \zeta (\ell +k)\,.
\]
This inequality also implies that $\ell + k \ge 1$.
Thus, $\alpha^{-1} \in P^{-1}$.
\end{proof}

\begin{proof}[Proof of Theorem \ref{thm3_2}]
We have $PP \subset P$ by Lemma \ref{lem3_4}, we have $G_1 P G_1 \subset P$ by Lemma \ref{lem3_5}, and we have the disjoint union $G = P \sqcup P^{-1} \sqcup G_1$ by Lemma \ref{lem3_6}, Lemma \ref{lem3_7} and Lemma \ref{lem3_8}.
\end{proof}


\section{Artin groups of type A}\label{Sec4}

In this section we assume that $G$ and $M$ are the Artin group and the Artin monoid of type $A_n$, respectively, where $n \ge 2$.
Recall that $G$ is defined by the presentation
\[
G = \langle s_1, \dots, s_n \mid s_i s_j s_i = s_j s_i s_j \text{ for } |i-j|=1,\ s_i s_j = s_j s_i \text{ for } |i-j| \ge 2 \rangle\,,
\]
and that $M$ is the submonoid of $G$ generated by $s_1, s_2, \dots, s_n$.
Recall also that $G$ is the braid group $\BB_{n+1}$ on $n+1$ strands and $M$ is the positive braid monoid $\BB_{n+1}^+$.
By Brieskorn--Saito \cite{BriSai1} and Deligne \cite{Delig1}, $(G,M, \Omega)$ is a Garside structure, where $\Omega = (s_1 \cdots s_n) \cdots (s_1 s_2 s_3) (s_1 s_2) s_1$.
The element $\Omega$ is not central in $G$ but $\Delta = \Omega^2 = (s_1 \cdots s_n)^{n+1}$ is central and, by Dehornoy \cite{Dehor3}, $(G,M,\Delta)$ is also a Garside structure on $G$.
The latter is the Garside structure that we consider in this section.

We denote by $G_1$ (resp. $M_1$) the subgroup of $G$ (resp. the submonoid of $M$) generated by $s_2, \dots, s_{n}$ and we set $\Delta_1 = (s_2 \cdots s_n)^n$.
Then  $(G_1, M_1, \Delta_1)$ is a parabolic substructure of $(G,M, \Delta)$ and $\Delta_1$ is central in $G_1$.
On the other hand, we denote by $H$ (resp. $N$) the subgroup of $G$ (resp. the submonoid of $M$) generated by $s_1, \dots, s_{n-1}$ and we set $\Lambda = (s_1 \cdots s_{n-1})^n$.
Again, $(H,N, \Lambda)$ is a parabolic substructure of $(G,M,\Delta)$.
Observe that $M_1 \cup N$ generates $M$.

The purpose of this section is to prove the following. 

\begin{thm}\label{thm4_1}
The pair $(H, G_1)$ satisfies Condition A with constant $\zeta = 1$ and Condition B with constant $\zeta = 1$.
\end{thm}

By applying Theorem \ref{thm3_2} we deduce the following.

\begin{corl}\label{corl4_2}
The pair $(H,G_1)$ is a Dehornoy structure.
\end{corl}

For $1 \le i \le n-1$ we set $G_i = \langle s_{i+1}, \dots, s_n \rangle$, $M_i = \langle s_{i+1}, \dots, s_n \rangle^+$, $\Delta_i = (s_{i+1} \cdots s_n)^{n+1-i}$ and $H_i = \langle s_i, \dots, s_{n-1} \rangle$.
By iterating Corollary \ref{corl4_2} and applying Proposition \ref{prop3_1} we get the following.

\begin{corl}\label{corl4_3}
\begin{itemize}
\item[(1)]
For each $1 \le i \le n-1$ the pair $(H_i, G_i)$ is a Dehornoy structure on $(G_{i-1}, M_{i-1}, \Delta_{i-1})$, where $(G_0, M_0, \Delta_0) = (G, M, \Delta)$.
\item[(2)]
For each $1 \le i \le n-1$ we denote by $P_i$ the set of $(H_i,G_i)$-positive elements of $G_{i-1}$.
Furthermore we set $P_n = \{ s_n^k \mid k \ge 1 \}$.
For each $\epsilon = (\epsilon_1, \dots, \epsilon_n) \in \{ \pm 1\}^n$ the set $P^\epsilon = P_1^{\epsilon_1} \sqcup \cdots \sqcup P_n^{\epsilon_n}$ is the positive cone for a left-order on $G$.
\end{itemize}
\end{corl}

Before proving Theorem \ref{thm4_1} we show that the orders on $G$ given in Corollary \ref{corl4_3}\,(2) coincide with those obtained using Theorem \ref{thm1_1}.
More precisely we prove the following.

\begin{prop}\label{prop4_4}
The set $P = P_{H,G_1}$ of $(H,G_1)$-positive elements is equal to the set of $s_1$-positive elements of $G = \BB_{n+1}$.
\end{prop}

\begin{proof}
Let $P'$ denote the set of $s_1$-positive elements of $G$.
We know by Dehornoy \cite{Dehor1} that we have the disjoint union $G = P' \sqcup P'^{-1} \sqcup G_1$.
We also know by Corollary \ref{corl4_2} that $P P \subset P$, $G_1 P G_1 \subset P$ and $G = P \sqcup P^{-1} \sqcup G_1$.
Let $\alpha \in P'$.
By definition $\alpha$ is written $\alpha = \alpha_0 s_1 \alpha_1 \cdots s_1 \alpha_p$ where $p \ge 1$ and $\alpha_0, \alpha_1, \dots, \alpha_p \in G_1$.
The $\Delta$-form of $s_1$ is $s_1 = s_1 \Delta^0$, hence $s_1$ does not lie in $P^{-1}$.
The element $s_1$ does not lie in $G_1$ either, hence $s_1$ lies in $P$.
Since $P P \subset P$ and $G_1 P G_1 \subset P$ we deduce that $\alpha$ lies in $P$.
So, $P' \subset P$ and therefore $P'^{-1} \subset P^{-1}$.
Since we have disjoint unions $G = P \sqcup P^{-1} \sqcup G_1$ and $G = P' \sqcup P'^{-1} \sqcup G_1$ we conclude that $P = P'$ and $P^{-1} = P'^{-1}$.
\end{proof}

The rest of the section is dedicated to the proof of Theorem \ref{thm4_1}.
We recall once for all the expressions of $\Delta$ and $\theta$ over the standard generators. 
\begin{gather*}
\Delta = (s_1 s_2 \cdots s_n)^{n+1} = (s_1 \cdots s_{n-1} s_n^2 s_{n-1} \cdots s_1) \cdots (s_{n-1} s_n^2 s_{n-1}) s_n^2\,, \\
\theta = s_1 \cdots s_{n-1} s_n^2 s_{n-1} \cdots s_1\,.
\end{gather*}

\begin{prop}\label{prop4_5}
The pair $(H, G_1)$ satisfies Condition A with constant $\zeta = 1$.
\end{prop}

\begin{proof}
Let $k \ge 1$.
Then, by Dehornoy \cite{Dehor2}, $\bh (\Delta^k) = \bh(\Omega^{2k}) = 2k+2$, hence $\dpt (\Delta^k) = k+1$.
\end{proof}

It remains to show that $(H,G_1)$ satisfies Condition B with constant $\zeta = 1$ (see Proposition \ref{prop4_12}).
This is the goal of the rest of the section.

An \emph{$(N,M_1)$-expression of length $p$} of an element $a \in M$ is defined to be an expression of $a$ of the form $a = a_p \cdots a_2 a_1$ with $a_i \in N$ if $i$ is even and $a_i \in M_1$ if $i$ is odd.

\begin{lem}[Dehornoy \cite{Dehor2}, Burckel \cite{Burck1}]\label{lem4_6}
Let $a \in M$ and let $a = a_p \cdots a_2 a_1$ be an $(N,M_1)$-expression of $a$.
Then $p \ge \bh (a)$.
\end{lem}

Let $a \in M$.
Choose an expression $a = s_{i_\ell} \cdots s_{i_2} s_{i_1}$ of $a$ over $S$ and set $\rev (a) = s_{i_1} s_{i_2} \cdots s_{i_\ell}$.
Since the relations that define $M$ are symmetric, the definition of $\rev (a)$ does not depend on the choice of the expression of $a$. 
It is easily checked that $\rev (\Omega)= \Omega$, $\rev(\Delta) = \Delta$ and $\rev (\theta) = \theta$.
Moreover, $\rev (a) \in M_1$ for all $a \in M_1$ and $\rev (a) \in N$ for all $a \in N$.

\begin{lem}\label{lem4_7}
Let $a \in M$.
Then $\dpt (\rev (a)) = \dpt(a)$.
\end{lem}

\begin{proof}
Let $a = a_p \cdots a_2 a_1$ be the alternating form of $a$. 
If $p$ is even, then $\rev(a) = \rev(a_1)\, \rev(a_2) \cdots \rev(a_p)\,1$ is a $(N,M_1)$-expression of $\rev(a)$ hence, by Lemma \ref{lem4_6}, $p+1 \ge \bh(\rev(a))$, and therefore $\dpt(a) = \frac{p}{2} \ge \dpt(\rev(a))$.
If $p$ is odd, then $\rev(a) = \rev(a_1)\, \rev(a_2) \cdots \rev(a_p)$ is a $(N,M_1)$-expression of $\rev(a)$ hence, by Lemma \ref{lem4_6}, $p \ge \bh (\rev (a))$, and therefore $\dpt(a) = \frac{p-1}{2} \ge \dpt (\rev (a))$.
So, $\dpt (a) \ge \dpt (\rev(a))$ in both cases. 
Since $\rev (\rev (a)) = a$, we also have $\dpt (\rev(a)) \ge \dpt(a)$, hence $\dpt (\rev(a)) = \dpt (a)$.
\end{proof}

\begin{lem}\label{lem4_8}
Let $a \in M \setminus M_1$ and $k \ge 1$.
Then $\dpt(a \theta^k) = \dpt(a) + k$.
\end{lem}

\begin{proof}
Let $a \in M \setminus M_1$.
It suffices to show that $\bh (a \theta) = \bh (a) + 2$.
Let $a = a_p \cdots a_2 a_1$ be the alternating form of $a$.
Note that, since $a \not\in M_1$, we have $p \ge 2$.
Note also that, by Lemma \ref{lem3_3}\,(1), we have $a_1 \theta = \theta a_1$.
Then 
$a \theta = a_p \cdots a_3 a_2 \theta a_1 = 
a_p \cdots a_3 b_4 b_3 b_2 a_1$,
where $b_4 = a_2 s_1 \in N$, $b_3 = s_2 \cdots s_{n-1} s_n^2 \in M_1$ and $b_2 = s_{n-1} \cdots s_2 s_1 \in N$.
We turn to show that $a \theta = a_p \cdots a_2 b_4 b_3 b_2 a_1$ is the alternating form of $a \theta$.
This will prove the lemma.

Let $x = \tau_{M_1} (a_p \cdots a_3 b_4 b_3 b_2) = \tau_{M_1} (a_p \cdots a_3 a_2 \theta)$.
We know by Lemma \ref{lem3_3}\,(1) that $x \vee_R \theta = \theta x = x \theta$, hence $x \le_R a_p \cdots a_2$, and therefore $x=1$, since $\tau_{M_1}(a_p \cdots a_3 a_2)=1$. 
We have $a_p \cdots a_3 b_4 b_3 = a_p \cdots a_3 a_2 s_1 s_2 \cdots s_{n-1} s_n^2$.
It is easily checked that $(s_1 \cdots s_{n-1} s_n^2) \vee_R s_i = s_{i+1}(s_1 \cdots s_{n-1} s_n^2)$ for all $i \in \{1, \dots, n-1\}$.
Thus, if there exists $i \in \{1, \dots, n-1\}$ such that $s_i \le_R a_p \cdots a_3 b_4 b_3$, then there exists $j \in \{2, \dots, n\}$ such that  $s_j \le_R a_p \cdots a_3 a_2$.
But, since $\tau_{M_1}(a_p \cdots a_3 a_2)=1$, such a $j$ does not exist, hence such an $i$ does not exist either, hence $\tau_N (a_p \cdots a_3 b_4 b_3) = 1$.
We have $a_p \cdots a_3 b_4 = a_p \cdots a_3 a_2 s_1$.
We have $s_1 \vee_R s_i = s_i s_1$ for all $i \in \{3, \dots, n\}$, and $s_1 \vee_R s_2 = s_1 s_2 s_1$.
Thus, for $i \in \{2, \dots, n\}$, if $s_i \le_R a_p \cdots a_3 b_4$, then $s_i \le_R a_p \cdots a_3 a_2$.
Since such an $i$ does not exist, we have $\tau_{M_1} (a_p \cdots a_3 b_4)=1$.
This finishes the proof that $a_p \cdots a_3 b_4 b_3 b_2 a_1$ is the alternating form of $a \theta$ since $a_p \cdots a_3$ is an alternating form and $\tau_N (a_p \cdots a_3) =1$.
\end{proof}

\begin{lem}\label{lem4_9}
\begin{itemize}
\item[(1)]
Let $a \in M_1$ and $b \in M \setminus M_1$.
Then $\dpt (ab) = \dpt(ba) = \dpt(b)$.
\item[(2)]
Let $a \in \Theta$ and $b \in M \setminus M_1$.
Then $\dpt (ab) = \dpt(ba) = \dpt(a) + \dpt(b) -1$.
\end{itemize}
\end{lem}

\begin{proof}
Let $a \in M_1$ and $b \in M \setminus M_1$.
We obviously have $\bh (ba) = \bh (b)$, hence $\dpt(ba) = \dpt(b)$.
On the other hand, since $\rev(a) \in M_1$, By Lemma \ref{lem4_7} we have $\dpt(ab) = \dpt (\rev (ab)) = \dpt (\rev(b)\, \rev(a)) = \dpt( \rev (b)) = \dpt(b)$.

Let $a \in \Theta$ and $b \in M \setminus M_1$.
Write $a = \theta^k a_0$ with $a_0 \in M_1$ and $k \ge 1$.
By the above and Proposition \ref{prop4_5} we have $\dpt(a) = \dpt(\theta^k) = \dpt(\Delta^k) = k+1$.
Then, by the above and Lemma \ref{lem4_8}, $\dpt(ba) = \dpt(b \theta^k) = \dpt(b) + k = \dpt(a) + \dpt(b) -1$.
On the other hand, since $\rev(a) \in \Theta$, we have $\dpt(ab) = \dpt(\rev(ab)) = \dpt (\rev(b)\, \rev(a)) = \dpt(\rev(a)) + \dpt(\rev(b)) -1 = \dpt(a) + \dpt(b)-1$.
\end{proof}

\begin{lem}\label{lem4_10}
Let $a \in M$ and $k \ge 0$.
If $a \Delta^{-k} \in G_1$ then $a \in \bar \Theta$.
\end{lem}

\begin{proof}
Let $a \Delta^{-k} = a_0 \Delta_1^{-t}$ be the $\Delta_1$-form of $a \Delta^{-k}$.
We have $a = a_0 \Delta_1^{-t} \Delta^k = \theta^k a_0 \Delta_1^{k-t}$.
If $k \ge t$ then we clearly have $a \in \bar \Theta$.  
Suppose that $k < t$.
Then $a \Delta_1^{t-k} = \theta^k a_0$, hence $\Delta_1^{t-k} \le_R \tau_{M_1} (\theta^k a_0)$.
By Lemma \ref{lem3_3}\,(1) we have $\tau_{M_1} (\theta^k a_0) = a_0$, hence $\Delta_1^{t-k} \le_R a_0$.
Let $b_0 \in M_1$ such that $a_0 = b_0 \Delta_1^{t-k}$.
Then $a = \theta^k b_0 \in \bar \Theta$.
\end{proof}

\begin{lem}\label{lem4_11}
Let $a,b \in M\setminus M_1$, $c \in M_1$ and $k \ge 0$ such that $ab = c \Delta^k$ and $\dpt(a) + \dpt(b) = k+2$.
Then $(a,b) \in (\Theta \times \Theta)$.
\end{lem}

\begin{proof}
Let $p = \dpt (a)$ and $q = \dpt (b)$.
Note that, since $a,b \not \in M_1$, we have $p,q \ge 1$.
We have $\bh (a) \ge 2p$, hence $\bh (a)-1 > 2p-2$, and therefore, by Theorem \ref{thm2_5}, $\Omega^{-2p+2} a = a \Delta^{-p+1}$ either lies in $G_1$ or is $s_1$-positive. 
Similarly, $b \Delta^{-q+1}$ either lies in $G_1$ or is $s_1$-positive. 
If either $a \Delta^{-p+1}$ was $s_1$-positive or $b \Delta^{-q+1}$ was $s_1$-positive, then $c = ab \Delta^{-k} = (a \Delta^{-p+1})(b \Delta^{-q+1})$ would be $s_1$-positive. 
Since $c \in M_1$, $c$ cannot be $s_1$-positive, hence both $a \Delta^{-p+1}$ and $b \Delta^{-q+1}$ lie in $G_1$.
We conclude by Lemma \ref{lem4_10} that $a,b \in \bar \Theta$, hence $a,b \in \Theta$ since we assumed that $a,b \not\in M_1$.
\end{proof}

Now we are ready to prove the second part of Theorem \ref{thm4_1}.

\begin{prop}\label{prop4_12}
The pair $(H, G_1)$ satisfies Condition B with constant $\zeta = 1$.
\end{prop}

\begin{proof}
We take $(a, b) \in (M \times M) \setminus (\bar \Theta \times \bar \Theta)$ such that $a$ and $b$ are unmovable.
We must show that $(a,b)$ satisfies Condition B with constant $\zeta = 1$.
Let $ab = c \Delta^t$ be the $\Delta$-form of $ab$.
So, we must show that there exists $\varepsilon \in \{ 0, 1\}$ such that $\dpt (c) = \dpt (a) + \dpt (b) - t - \varepsilon$, and $\varepsilon = 1$ if either $a \in \Theta$, or $b \in \Theta$, or $c \in M_1$.

{\it Case 1: $a \in M_1$ and $b \in M \setminus \bar \Theta$.}
By Lemma \ref{lem3_3}\,(6) we have $ab \not \in \bar \Theta$, hence $c \not \in M_1$.
Then, by Lemma \ref{lem4_9}, $\dpt(a) + \dpt (b) = \dpt (b) = \dpt (a b) = \dpt (c) + t$, hence $\dpt(c) = \dpt(a) + \dpt (b) -t -0$.
The case $a \in M \setminus \bar \Theta$ and $b \in M_1$ is proved in a similar way.

{\it Case 2: $a \in \Theta$ and $b \in M \setminus \bar \Theta$.}
We write $a = \theta^k a_0$ where $k \ge 1$ and $a_0 \in M_1$.
Again, by Lemma \ref{lem3_3}\,(6) we have $ab \not \in \bar \Theta$, hence $c \not\in M_1$.
Then, by Lemma \ref{lem4_9}, $\dpt (a) + \dpt(b) - 1 = \dpt (ab) = \dpt (c) +t$, hence $\dpt (c) = \dpt(a) + \dpt (b) - t - 1$.
The case $a \in M \setminus \bar \Theta$ and $b \in \Theta$ is proved in a similar way.

{\it Case 3: $a,b \in M \setminus \bar \Theta$.}
Set $p = \dpt (a)$ and $q = \dpt (b)$.
We have $\bh (a) \in \{ 2p, 2p+1 \}$ hence, by Theorem \ref{thm2_5}, $\Omega^{-2p} a$ is $s_1$-negative and $\Omega^{-2p+2} a$ either lies in $G_1$ or is $s_1$-positive. 
Similarly, $\Omega^{-2q} b$ is $s_1$-negative and $\Omega^{-2q+2} b$ either lies in $G_1$ or is $s_1$-positive.
So, $\Omega^{-2p-2q} ab$ is $s_1$-negative and $\Omega^{-2p-2q+4} ab$ either lies in $G_1$ or is $s_1$-positive. 
By Theorem \ref{thm2_5} it follows that $\bh(ab) - 1 \le 2p + 2q$ and $2p + 2q - 4 < \bh (ab)-1$, hence $2p + 2q - 2 \le \bh (ab) \le 2p + 2q + 1$, and therefore $p+q - 1 \le \dpt (ab) \le p+q$.
So, there exists $\varepsilon \in \{ 0, 1 \}$ such that $\dpt (ab) = p + q - \varepsilon = \dpt (a) + \dpt (b) - \varepsilon$.

Suppose that $c \not\in M_1$.
By Lemma \ref{lem4_9}\,(2), $\dpt (c) + t = \dpt(c) + \dpt (\Delta^t) - 1 = \dpt (c \Delta^t) = \dpt (ab) = \dpt(a) + \dpt(b) - \varepsilon$, hence $\dpt (c) = \dpt (a) + \dpt (b) - t - \varepsilon$.
Suppose that $c \in M_1$.
By Lemma \ref{lem4_9}\,(1), $\dpt (a) + \dpt (b) - \varepsilon = \dpt (ab) = \dpt (c \Delta^t) = \dpt (\Delta^t) = t+1$, hence $\dpt(a) + \dpt(b) = t + 1 + \varepsilon$.
Since $a,b \not\in \bar \Theta$ Lemma \ref{lem4_11} implies that $\varepsilon=0$.
So, $\dpt (c) = 0 = \dpt (a) + \dpt (b) -t -1$.
\end{proof}


\section{Artin groups of dihedral type, the even case}\label{Sec5}

Let $m \ge 4$ be an integer.
Recall that the \emph{Artin group of type $I_2(m)$} is the group $G = A_{I_2(m)}$ defined by the presentation $G = \langle s,t \mid \Pi (s,t,m) = \Pi (t,s,m) \rangle$.
Let $M$ be the submonoid of $G$ generated by $\{s,t\}$ and let $\Omega = \Pi (s, t, m)$.
Then, by Brieskorn--Saito \cite{BriSai1} and Deligne \cite{Delig1}, the triple $(G, M, \Omega)$ is a Garside structure on $G$.
If $m$ is even then $\Delta = \Omega$ is central. 
However, if $m$ is odd then $\Omega$ is not central but $\Delta = \Omega^2$ is central. 
In both cases, by Dehornoy \cite{Dehor3}, the triple $(G,M, \Delta)$ is a Garside structure on $G$.
In this section we study the case where $m$ is even and in the next one we will study the case where $m$ is odd.
So, from now until the end of the section we assume that $m = 2k$ is even and $\Delta = \Pi(s,t,m) = (st)^k = (ts)^k$.

\begin{rem}
By setting $\Delta = \Omega^2$ in the even case as in the odd case we could state global results valid for all $m \ge 4$, but it would be still necessary to differentiate the even case from the odd case in the proofs, and this would lengthen the proofs for the even case. 
\end{rem}

We denote by $G_1$ (resp. $M_1$) the subgroup of $G$ (resp. submonoid of $M$) generated by $t$, and by $H$ (resp. $N$) the subgroup of $G$ (resp. submonoid of $M$) generated by $s$.
We set $\Delta_1 = t$ and $\Lambda = s$.
By Brieskorn--Saito \cite{BriSai1} the triples $(G_1, M_1, \Delta_1)$ and $(H, N, \Lambda)$ are parabolic substructures of $(G, M, \Delta)$.
On the other hand it is obvious that $M_1 \cup N$ generates $M$.
The main result of the present section is the following.

\begin{thm}\label{thm5_1}
The pair $(H,G_1)$ satisfies Condition A with constant $\zeta = k-1$ and Condition B with constant $\zeta = k-1$.
\end{thm}

By Theorem \ref{thm3_2} this implies the following.

\begin{corl}\label{corl5_2}
The pair $(H, G_1)$ is a Dehornoy structure on $G$.
\end{corl}

We denote by $P_1$ the set of $(H, G_1)$-positive elements of $G$ and we set $P_2 = \{ t^n \mid n \ge 1 \}$.
For each $\epsilon = ( \epsilon_1, \epsilon_2) \in \{ \pm 1\}^2$ we set $P^\epsilon = P_1^{\epsilon_1} \cup P_2^{\epsilon_2}$.
Then, by Proposition \ref{prop3_1}, we have the following.

\begin{corl}\label{corl5_3}
The set $P^\epsilon$ is the positive cone for a left-order on $G$.
\end{corl}

In this section we denote by $r_1, \dots, r_{2k-1}$ the standard generators of the braid group $\BB_{2k}$ on $2k=m$ strands.
By Crisp \cite{Crisp1} we have an embedding $\iota : G \to \BB_{2k}$ which sends $s$ to $\prod_{i=0}^{k-1} r_{2i+1}$ and sends $t$ to $\prod_{i=1}^{k-1} r_{2i}$.
In the second part of the section we will show that the orders obtained from Corollary \ref{corl5_3} can be deduced from $\iota$ together with the Dehornoy order. 
More precisely, we show the following.

\begin{prop}\label{prop5_4}
Let $\alpha \in G$.
Then $\alpha$ is $(H, G_1)$-negative if and only if $\iota (\alpha)$ is $r_1$-negative.
\end{prop}

The proof of Theorem \ref{thm5_1} is based on the following observation whose proof is left to the reader.

\begin{lem}\label{lem5_5}
Let $a$ be an unmovable element of $M$.
Then $a$ is uniquely written in the form $a = t^{u_p} s^{v_p} \cdots t^{u_1} s^{v_1} t^{u_0}$ with $u_0, u_p \ge 0$, $u_1, \dots, u_{p-1} \ge 1$ and $v_1, \dots, v_p \ge 1$.
In this case $\dpt(a) = p$.
\end{lem}

The first part of Theorem \ref{thm5_1} is a straightforward consequence of this lemma.

\begin{prop}\label{prop5_6}
The pair $(H, G_1)$ satisfies Condition A with constant $\zeta = k-1$.
\end{prop}

\begin{proof}
Let $p \ge 1$ be an integer.
We have $\theta = s(ts)^{k-1}$, hence $\theta^p = (s(ts)^{k-1})^p$.
By Lemma \ref{lem5_5} it follows that $\dpt (\theta^p) = p(k-1)+1$, hence $\dpt (\Delta^p) = \dpt( \theta^p t^p) = \dpt (\theta^p) = p(k-1)+1$.
\end{proof}

If $a \in M\setminus \{1\}$ is written as in Lemma \ref{lem5_5} we set $\sigma (a) = t$ if $u_p \neq 0$ and $\sigma (a) = s$ if $u_p=0$.
Similarly we set $\tau(a) = t$ if $u_0 \neq 0$ and $\tau(a) = s$ if $u_0 = 0$.
In other words $\sigma (a)$ is the first letter of $a$ and $\tau (a)$ is the last one. 
The following is a straightforward consequence of Lemma \ref{lem5_5}.

\begin{lem}\label{lem5_7}
\begin{itemize}
\item[(1)]
Let $a,b$ be two unmovable elements of $M$ such that $ab$ is unmovable. 
Then
\[
\dpt(ab) = \left\{ \begin{array}{ll}
\dpt(a) + \dpt(b) - 1 & \text{if } a\neq 1,\ b\neq 1 \text{ and } \tau (a) = \sigma (b) = s\,,\\
\dpt(a) + \dpt(b) &\text{otherwise}\,.
\end{array} \right.
\]
\item[(2)]
Let $a,b \in M$ such that $ab = \Delta$.
Then $\dpt(a) + \dpt(b) = \dpt(\Delta) = k$.
\end{itemize}
\end{lem}

Now we can prove the second part of Theorem \ref{thm5_1}.

\begin{prop}\label{prop5_8}
The pair $(H, G_1)$ satisfies Condition B with constant $\zeta = k-1$.
\end{prop}

\begin{proof}
We take two unmovable elements $a,b \in M$ such that $(a,b) \not\in \bar\Theta \times \bar\Theta$ and we denote by $ab = c \Delta^p$ the $\Delta$-form of $ab$.
We must show that there exists $\varepsilon \in \{0,1\}$ such that $\dpt (c) = \dpt (a) + \dpt (b) -p(k-1) - \varepsilon$ and that $\varepsilon = 1$ if either $a \in \Theta$, or $b \in \Theta$, or $c \in M_1$.
We write $a = a_{p+1} a_p \cdots a_1$ and $b = b_1 \cdots b_p b_{p+1}$ so that:
\begin{itemize}
\item
$a_i \neq 1$, $b_i \neq 1$ and $a_i b_i = \Delta$ for all $i \in \{1, \dots, p\}$;
\item
$a_{p+1} b_{p+1} = c$;
\item
We set $x_i = \tau (a_i)$, $x_i' = \sigma (a_i)$, $y_i = \sigma (b_i)$, $y_i' = \tau (b_i)$ for all $i \in \{1, \dots, p+1\}$.
Then $x_i' = x_{i+1}$ for all $i \in \{1, \dots, p-1\}$.
\end{itemize}
We denote by $\varphi: M \to M$ the isomorphism that sends $s$ to $t$ and $t$ to $s$.
Since $a_i b_i = \Delta$, we have $y_i = \varphi(x_i)$ and $y_i' = \varphi (x_i')$ for all $i \in \{1, \dots, p\}$.
In particular, $y_i' = \varphi (x_i') = \varphi (x_{i+1}) = y_{i+1}$ for all $i \in \{1, \dots ,p-1\}$.

Let $u = |\{i \in \{1, \dots, p\} \mid  x_i' = s \}|$.
By Lemma \ref{lem5_7}, $\dpt (a) = \dpt (a_{p+1}) + \sum_{i=1}^p \dpt(a_i) - u + \varepsilon_a$, where $\varepsilon_a$ is as follows. 
If $p \ge 1$ and $a_{p+1} \neq 1$, then: $\varepsilon_a = 0$ if $(x_p',x_{p+1}) \in \{ (s,s), (t,s), (t,t) \}$ and $\varepsilon_a = 1$ if $(x_p',x_{p+1}) = (s,t)$.
If $p \ge 1$ and $a_{p+1} = 1$, then: $\varepsilon_a = 0$ if $x_p'=t$ and $\varepsilon_a = 1$ if $x_p'=s$.
If $p=0$, then $\varepsilon_a = 0$.

Let $v = |\{i \in \{1, \dots, p\} \mid y_i' = s \}|$.
As for $a$, by applying Lemma \ref{lem5_7} we obtain $\dpt (b) = \dpt (b_{p+1}) + \sum_{i=1}^p \dpt (b_i) - v + \varepsilon_b$ where $\varepsilon_b$ is as follows. 
If $p \ge 1$ and $b_{p+1} \neq 1$, then: $\varepsilon_b = 0$ if $(y_p',y_{p+1}) \in \{ (s,s), (t,s), (t,t) \}$ and $\varepsilon_b = 1$ if $(y_p',y_{p+1}) = (s,t)$.
If $p \ge 1$ and $b_{p+1} = 1$, then: $\varepsilon_b = 0$ if $y_p'=t$ and $\varepsilon_b = 1$ if $y_p'=s$.
If $p=0$, then $\varepsilon_b = 0$.

By applying again Lemma \ref{lem5_7} we obtain $\dpt (c) = \dpt (a_{p+1}) + \dpt (b_{p+1}) + \varepsilon_c$ where $\varepsilon_c$ is as follows. 
If $a_{p+1} \neq 1$ and $b_{p+1}\neq 1$, then: $\varepsilon_c = -1$ if $(x_{p+1}, y_{p+1}) = (s,s)$ and $\varepsilon_c = 0$ if  $(x_{p+1}, y_{p+1}) \in \{ (s,t), (t,s), (t,t) \}$.
If either $a_{p+1} = 1$ or $b_{p+1} = 1$, then $\varepsilon_c = 0$.

Finally, by Lemma \ref{lem5_7}\,(2), we have $\sum_{i=1}^p (\dpt (a_i) + \dpt (b_i) ) = pk$.
On the other hand, since $y_i' = \varphi (x_i')$ for all $i \in \{1, \dots, p\}$, we have $u+v = p$.

Set $\varepsilon = \varepsilon_a + \varepsilon_b - \varepsilon_c$.
By the above we have $\dpt (c) = \dpt (a) + \dpt (b) -p(k-1) - \varepsilon$ and $\varepsilon$ is as follows. 
If $p \ge 1$, $a_{p+1} \neq 1$ and $b_{p+1} \neq 1$, then: $\varepsilon = 0$ if $(x_{p}', x_{p+1}, y_{p+1}) \in \{ (s,s,t), (t,t,s) \}$ and $\varepsilon = 1$ otherwise.
If $p \ge 1$, $a_{p+1} \neq 1$ and $b_{p+1} =1$, then: $\varepsilon = 0$ if $(x_{p}',x_{p+1}) = (s,s)$ and $\varepsilon = 1$ otherwise.
If $p \ge 1$, $a_{p+1} = 1$ and $b_{p+1} \neq 1$, then: $\varepsilon = 0$ if $(x_{p}',y_{p+1}) = (t,s)$ and $\varepsilon = 1$ otherwise.
If $p \ge 1$, $a_{p+1} = 1$ and $b_{p+1} = 1$, then $\varepsilon = 1$.
If $p = 0$, $a \neq 1$ and $b \neq 1$, then: $\varepsilon = 0$ if $(x_{p+1}, y_{p+1}) \in \{ (s,t), (t,s), (t,t) \}$ and $\varepsilon = 1$ otherwise.
If $p=0$ and either $a=1$ or $b=1$, then $\varepsilon = 0$.

Suppose that $a \in \Theta$.
Then $a$ is written $a = \theta^q$ with $q \ge 1$.
Set $b = t^r b'$ where $b' \neq 1$ (since $b \not\in \bar\Theta$) and $\sigma (b') = s$.
If $r=0$, then $p=0$, $a = \theta^q \neq 1$, $b = b' \neq 1$ and $(x_{p+1},y_{p+1}) = (s,s)$, hence $\varepsilon = 1$.
If $0 < r < q$, then $r=p > 0$, $a_{p+1} = \theta^{q-p} \neq 1$, $b_{p+1} = b' \neq 1$ and $(x_{p}', x_{p+1}, y_{p+1}) = (s,s,s)$, hence $\varepsilon = 1$.
If $r=q$, then $r=p=q$, $a_{p+1}=1$, $b_{p+1} = b' \neq 1$ and $(x_{p}',y_{p+1}) = (s,s)$, hence $\varepsilon = 1$.
If $r > q$, then $p = q$, $a_{p+1} = 1$, $b_{p+1}=t^{r-q} b'$ and $(x_{p}',y_{p+1}) = (s,t)$, hence $\varepsilon = 1$.
The case $b \in \Theta$ is proved in the same way.

Suppose that $c \in M_1$.
Then $p \ge 1$, since $(a,b) \not\in (\bar\Theta \times \bar\Theta)$.
If $a_{p+1} \neq 1$ and $b_{p+1} \neq 1$, then $(x_{p+1}, y_{p+1}) =(t,t)$, hence $\varepsilon = 1$.
If $a_{p+1} \neq 1$ and $b_{p+1} =1$, then $x_{p+1} =t$, hence $\varepsilon = 1$.
If $a_{p+1} = 1$ and $b_{p+1} \neq 1$, then $y_{p+1} = t$, hence $\varepsilon = 1$.
If $a_{p+1} = 1$ and $b_{p+1} = 1$, then $\varepsilon = 1$.
\end{proof}

We turn now to the proof of Proposition \ref{prop5_4}.
We denote by $G_1'$ (resp $M_1'$) the subgroup of $\BB_{2k}$ (resp. the submonoid of $\BB_{2k}^+$) generated by $r_2, \dots, r_{2k-1}$ and we denote by $H'$ (resp. $N'$) the subgroup of $\BB_{2k}$ (resp. the submonoid of $\BB_{2k}^+$) genetared by $r_1, \dots, r_{2k-2}$.
Note that $\iota (t) \in G_1'$, hence $\iota (G_1) \subset G_1'$.
We denote by $\Omega_\BB = (r_1 r_2 \cdots r_{2k-1}) \cdots (r_1 r_2) r_1$ the standard Garside element of $\BB_{2k}$ and by $\Phi : \BB_{2k} \to \BB_{2k}$, $\alpha \mapsto \Omega_\BB \alpha \Omega_\BB^{-1}$, the conjugation by $\Omega_\BB$.
Recall that $\Phi (r_i) = r_{2k-i}$ for all $i \in \{1, \dots, 2k-1\}$.
So, $\Phi(G_1') = H'$ and $\Phi (H') = G_1'$.

\begin{lem}\label{lem5_9}
Let $a$ be an unmovable element of $M$ such that $\dpt (a) \le k-1$.
Then there exist $b_1 \in M_1'$ and $b_2 \in N'$ such that $\iota (a) = b_1 b_2$.
\end{lem}

\begin{proof}
Let $p = \dpt (a)$.
By Lemma \ref{lem5_5}, $a$ can be written $a = t^{u_0} s^{v_1} t^{u_1} \cdots s^{v_p} t^{u_p}$ where $u_0, u_p \ge 0$, $u_1, \dots, u_{p-1} \ge 1$ and $v_1, \dots, v_p \ge 1$.
We show by induction on $p$ that there exist $b_1 \in M_1'$ and $b_2 \in \langle r_1, \dots, r_{2p} \rangle^+$ such that $\iota (a) = b_1 b_2$.
Since $p \le k-1$ this proves the lemma. 
The case $p=0$ is obvious because $\iota (t) \in M_1'$.
We assume that $1 \le p \le k-1$ and that the inductive hypothesis holds.
Set $a' = t^{u_0} s^{v_1} t^{u_1} \cdots s^{v_{p-1}} t^{u_{p-1}}$.
By induction there exist $b_1' \in M_1'$ and $b_2' \in \langle r_1, \dots, r_{2p-2} \rangle^+$ such that $\iota (a') = b_1' b_2'$.
Note that $b_2'$ commutes with $r_i$ for all $i \ge 2p$.
So, 
\begin{gather*}
\iota(a) = 
b_1' b_2' \left( \prod_{i=0}^{k-1} r_{2i+1}^{v_p} \right) \left( \prod_{i=1}^{k-1} r_{2i}^{u_p} \right) =
b_1'\left( \prod_{i=p}^{k-1} r_{2i+1}^{v_p} \right) b_2'\left( \prod_{i=0}^{p-1} r_{2i+1}^{v_p} \right) \left( \prod_{i=1}^{k-1} r_{2i}^{u_p} \right) =\\
b_1'\left( \prod_{i=p}^{k-1} r_{2i+1}^{v_p} \right) \left( \prod_{i=p+1}^{k-1} r_{2i}^{u_p} \right) b_2'\left( \prod_{i=0}^{p-1} r_{2i+1}^{v_p} \right) \left( \prod_{i=1}^{p} r_{2i}^{u_p} \right) = b_1 b_2\,,
\end{gather*}
where
\begin{gather*}
b_1 = b_1'\left( \prod_{i=p}^{k-1} r_{2i+1}^{v_p} \right) \left( \prod_{i=p+1}^{k-1} r_{2i}^{u_p} \right) \in M_1'\,,\\
b_2 = b_2'\left( \prod_{i=0}^{p-1} r_{2i+1}^{v_p} \right) \left( \prod_{i=1}^{p} r_{2i}^{u_p} \right) \in \langle r_1, \dots, r_{2p} \rangle^+\,.
\end{gather*}
\end{proof}

\begin{proof}[Proof of Proposition \ref{prop5_4}]
We denote by $P$ the set of $(H, G_1)$-positive elements of $G$ and by $P'$ the set of $r_1$-positive elements of $\BB_{2k}$.
By Corollary \ref{corl5_2} we have the disjoint union $G = P \sqcup P^{-1} \sqcup G_1$ and by Dehornoy \cite{Dehor1} we have the disjoint union $\BB_{2k} = P' \sqcup P'^{-1} \sqcup G_1'$.
It suffices to show that $\iota (P^{-1}) \subset P'^{-1}$.
Indeed, suppose that $\iota (P^{-1}) \subset P'^{-1}$.
Since $\iota$ is a homomorphism we also have $\iota (P) \subset P'$.
Since we also know that $\iota (G_1) \subset G_1'$, from the disjoint unions given above follows that $\alpha \in P^{-1}$ if and only if $\iota (\alpha) \in P'^{-1}$.

Let $\alpha$ be an element of $P^{-1}$.
Let $\alpha = a \Delta^{-p}$ be the $\Delta$-form of $\alpha$.
By definition we have $p \ge 1$ and $\dpt (a) \le p (k-1)$.
Suppose first that $p=1$ and $\dpt (a) \le k-1$.
By Lemma \ref{lem5_9} there exist $b_1 \in M_1'$ and $b_2 \in N'$ such that $\iota (a) = b_1 b_2$.
Moreover, by Crisp \cite{Crisp1}, $\iota (\Delta) = \Omega_\BB$.
Thus $\iota (\alpha) = b_1 b_2 \Omega_\BB^{-1} = b_1 \Omega_\BB^{-1} \Phi(b_2)$.
Since $b_1, \Phi(b_2) \in M_1'$ and $\Omega_\BB^{-1} \in P'^{-1}$, it follows that $\iota (\alpha) \in P'^{-1}$.

Now we consider the general case where $p \ge 1$ and $\dpt(a) \le p(k-1)$.
It is easily deduced from Lemma \ref{lem5_5} that $a$ can be written $a = a_1 a_2 \cdots a_p$ where $a_i$ is an unmovable element of $M$ such that $\dpt (a_i) \le k-1$ for all $i \in \{1, \dots, p\}$.
Note that $a_i$ may be equal to $1$ in the above expression. 
We have $\alpha = (a_1 \Delta^{-1}) (a_2 \Delta^{-1}) \cdots (a_p \Delta^{-1})$ and, by the above, $\iota (a_i \Delta^{-1}) \in P'^{-1}$ for all $i \in \{1, \dots, p\}$, hence $\iota (\alpha) \in P'^{-1}$.
\end{proof}


\section{Artin groups of dihedral type, the odd case}\label{Sec6}

Let $m= 2k+1 \ge 5$ be an odd integer and let $G = A_{I_2(m)} = \langle s,t \mid \Pi (s,t,m) = \Pi (t,s,m) \rangle$ be the Artin group of type $I_2 (m)$.
Let $M$ be the submonoid of $G$ generated by $\{s,t\}$ and let $\Omega = \Pi (s,t,m) = (st)^ks = (ts)^kt$.
Recall that, by Brieskorn--Saito \cite{BriSai1} and Deligne \cite{Delig1}, the triple $(G, M, \Omega)$ is a Garside structure on $G$. 
As pointed out in Section \ref{Sec5}, $\Omega$ is not central but $\Delta = \Omega^2$ is, and, by Dehornoy \cite{Dehor3}, $(G, M, \Delta)$ is also a Garside structure on $G$.
This is the Garside structure on $G$ that will be considered in the present section.

We denote by $G_1$ (resp. $M_1$) the subgroup of $G$ (resp. submonoid of $M$) generated by $t$, and by $H$ (resp. $N$) the subgroup of $G$ (resp. submonoid of $M$) generated by $s$.
Set $\Delta_1 = t^2$ and $\Lambda = s^2$.
Then, by Brieskorn--Saito \cite{BriSai1}, the triples $(G_1, M_1, \Delta_1)$ and $(H, N, \Lambda)$ are parabolic substructures of $(G, M, \Delta)$.
Moreover, $M_1 \cup N$ obviously generates $M$.
The main result of this section is the following.

\begin{thm}\label{thm6_1}
The pair $(H, G_1)$ satisfies Condition A with constant $\zeta = 2k-1$ and Condition B with constant $\zeta = 2k-1$.
\end{thm}

By Theorem \ref{thm3_2} this implies the following.

\begin{corl}\label{corl6_2}
The pair $(H, G_1)$ is a Dehornoy structure on $G$.
\end{corl}

We denote by $P_1$ the set of $(H, G_1)$-positive elements of $G$ and we set $P_2 = \{t^n \mid n \ge 1 \}$.
For each $\epsilon = (\epsilon_1, \epsilon_2) \in \{ \pm 1 \}^2$ we set $P^\epsilon = P_1 ^{\epsilon_1} \cup P_2^{\epsilon_2}$.
Then by Proposition \ref{prop3_1} we have the following.

\begin{corl}\label{corl6_3}
The set $P^\epsilon$ is the positive cone for a left-order on $G$.
\end{corl}

Let $r_1, \dots, r_{2k}$ be the standard generators of the braid group $\BB_{2k+1}$ on $m = 2k+1$ strands.
Again, by Crisp \cite{Crisp1}, we have an embedding $\iota : G \to \BB_{2k+1}$ which sends $s$ to $\prod_{i=0}^{k-1} r_{2i+1}$ and $t$ to $\prod_{i = 1}^k r_{2i}$.
The proof of the following is substantially the same as the proof of Proposition \ref{prop5_4}, hence it is left to the reader.

\begin{prop}\label{prop6_4} 
Let $\alpha \in G$.
Then $\alpha$ is $(H, G_1)$-negative if and only if $\iota (\alpha)$ is $r_1$-negative.
\end{prop}

We start now the proof of Theorem \ref{thm6_1}.
We say that an element $a \in M$ is \emph{$\Omega$-unmovable} if $\Omega \not \le_L a$ or, equivalently, if $\Omega \not \le_R a$.
The following is an observation.

\begin{lem}\label{lem6_5}
\begin{itemize} 
\item[(1)]
Let $a$ be an $\Omega$-unmovable element of $M$.
Then $a$ is uniquely written in the form $a = t^{u_p} s^{v_p}  \cdots t^{u_1} s^{v_1} t^{u_0}$, where $u_0, u_p \ge 0$, $u_1, \dots, u_{p-1} \ge 1$ and $v_1, \dots, v_p \ge 1$.
In this case we have $\dpt (a) = p$.
\item[(2)]
Let $a$ be an unmovable element of $M$.
Then $a$ is uniquely written in the form $a =a' \Omega^{\varepsilon}$ where $a'$ is $\Omega$-unmovable and $\varepsilon \in \{ 0, 1 \}$.
\end{itemize}
\end{lem}

The first part of Theorem \ref{thm6_1} is a direct consequence of this lemma.

\begin{prop}\label{prop6_6}
The pair $(H, G_1)$ satisfies Condition A with constant $\zeta = 2k-1$.
\end{prop}

\begin{proof}
Let $p \ge 1$ be an integer.
We have $\theta = (st)^k (ts)^k$, hence $\theta^p = ((st)^k(ts)^k)^p$.
By Lemma \ref{lem6_5}\,(1) it follows that $\dpt (\theta^p) = p(2k-1) +1$, hence $\dpt (\Delta^p) = \dpt(\theta^p t^{2p}) = \dpt(\theta^p) = p (2k-1) +1$.
\end{proof}

The second part of Theorem \ref{thm6_1} will be much more difficult to prove.
Let $a \in M\setminus \{1\}$ be an $\Omega$-unmovable element that we write as in Lemma \ref{lem6_5}\,(1).
Then we set $\sigma (a) = t$ if $u_p \neq 0$ and $\sigma (a) =s$ if $u_p =0$.
Similarly, we set $\tau (a) = t$ if $u_0 \neq 1$ and $\tau(a) =s$ if $u_0 =0$.
In other words, $\sigma(a)$ is the first letter of $a$ and $\tau (a)$ is its last one. 
On the other hand, we denote by $\varphi : G \to G$ the automorphism which sends $s$ to $t$ and $t$ to $s$.
Note that $\varphi$ is the conjugation by $\Omega$, that is, $\varphi (\alpha) = \Omega \alpha \Omega^{-1}$ for all $\alpha \in G$.
The following is again a direct consequence of Lemma \ref{lem6_5}.

\begin{lem}\label{lem6_7}
\begin{itemize} 
\item[(1)]
Let $a,b \in M$ such that $ab$ is $\Omega$-unmovable. 
Then 
\[
\dpt(ab) = \left\{ \begin{array}{ll}
\dpt (a) + \dpt (b) - 1 & \text{if } a \neq 1,\ b \neq 1, \text{ and }\tau (a) = \sigma (b) =s\,,\\
\dpt (a) + \dpt (b) & \text{otherwise}\,.
\end{array} \right.
\]
\item[(2)]
Let $a,b \in M \setminus \{ 1 \}$ such that $ab = \Omega$.
Then 
\[
\dpt(a) + \dpt(b) = \left\{ \begin{array}{ll}
k+1 &\text{if } \sigma (a) = s\,,\\
k &\text{if } \sigma (a) = t\,.
\end{array} \right.
\]
\item[(3)]
Let $c$ be an $\Omega$-unmovable element of $M$.
Then 
\[
\dpt(\varphi (c)) = \left\{ \begin{array}{ll}
\dpt(c) + 1 &\text{if } c\neq 1 \text{ and } \sigma(c) = \tau(c) = t\,,\\
\dpt(c) - 1 &\text{if }c \neq 1 \text{ and } \sigma(c) = \tau(c) = s\,,\\
\dpt(c) &\text{otherwise}\,.
\end{array} \right.
\]
\item[(4)]
Let $c$ be an $\Omega$-unmovable element of $M$.
Then 
\[
\dpt(c \Omega) = \left\{ \begin{array}{ll}
\dpt (c) + k - 1 &\text{if } c\neq 1 \text{ and } \tau (c) = s\,,\\
\dpt (c) + k &\text{otherwise}\,.
\end{array} \right.
\]
\item[(5)]
Let $a,b \in M \setminus \{ 1 \}$ such that $a\, \varphi(b) = \Omega$.
Then 
\[
\dpt(a) + \dpt(b) = \left\{ \begin{array}{ll}
k+1 &\text{if } \tau (a) = s\,,\\
k &\text{if } \tau (a) = t\,.
\end{array} \right.
\]

\end{itemize}
\end{lem}

\begin{lem}\label{lem6_8}
Let $a_1, a_2, b_1 ,b_2$ be four non-trivial $\Omega$-unmovable elements of $M$ such that $\sigma(a_1) = \tau(a_2)$, $\tau (b_1) = \sigma (b_2)$, $a_1 b_1 = \Omega$ and $a_2\, \varphi (b_2) = \Omega$.
Set $u = |\{ i \in \{1,2\} \mid \sigma(a_i)=s \}|$ and $v = |\{ i \in \{1,2\} \mid \tau(b_i)=s \}|$.
Then $\dpt (a_1) + \dpt (a_2) + \dpt (b_1) + \dpt (b_2) = 2k -1 + u + v$.
\end{lem}

\begin{proof}
If $\sigma (a_1) =s$ and $\sigma (a_2) = s$, then $\tau(a_2) = s$, $\tau (b_1) =s$ and $\tau (b_2) = t$, hence $u=2$, $v=1$ and, by Lemma \ref{lem6_7}, $\dpt (a_1) + \dpt (a_2) + \dpt (b_1) + \dpt (b_2) = 2k + 2 = 2k-1 + u + v$.
If $\sigma (a_1) =s$ and $\sigma (a_2) = t$, then $\tau(a_2) = s$, $\tau (b_1) =s$ and $\tau (b_2) = s$, hence $u=1$, $v=2$ and, by Lemma \ref{lem6_7}, $\dpt (a_1) + \dpt (a_2) + \dpt (b_1) + \dpt (b_2) = 2k + 2 = 2k-1 + u + v$.
If $\sigma (a_1) =t$ and $\sigma (a_2) = s$, then $\tau(a_2) = t$, $\tau (b_1) =t$ and $\tau (b_2) = t$, hence $u=1$, $v=0$ and, by Lemma \ref{lem6_7}, $\dpt (a_1) + \dpt (a_2) + \dpt (b_1) + \dpt (b_2) = 2k = 2k-1 + u + v$.
If $\sigma (a_1) =t$ and $\sigma (a_2) = t$, then $\tau(a_2) = t$, $\tau (b_1) =t$ and $\tau (b_2) = s$, hence $u=0$, $v=1$ and, by Lemma \ref{lem6_7}, $\dpt (a_1) + \dpt (a_2) + \dpt (b_1) + \dpt (b_2) = 2k = 2k-1 + u + v$.
\end{proof}

\begin{lem}\label{lem6_9}
Let $a,b$ be two $\Omega$-unmovable elements in $M$.
We assume that the $\Delta$-form of $ab$ is in the form $ab = c \Delta^p$ where $c$ is $\Omega$-unmovable. 
\begin{itemize}
\item[(1)]
Suppose that $(a, b) \not \in (\bar \Theta \times \bar \Theta)$.
There exists $\varepsilon \in \{0, 1 \}$ such that $\dpt (c) = \dpt (a) + \dpt (b) - p(2k-1) - \varepsilon$.
Moreover, $\varepsilon = 1$ if either $a \in \Theta$ or $b \in \Theta$ or $c \in M_1$.
\item[(2)]
Suppose that $(a \Omega, \varphi (b)) \not \in ( \bar \Theta \times \bar \Theta)$.
The exists $\varepsilon \in \{0, 1\}$ such that $\dpt (c \Omega) = \dpt (a \Omega) + \dpt (\varphi (b)) - p(2k-1) - \varepsilon$.
Moreover, $\varepsilon = 1$ if either $a \Omega \in \Theta$ or $\varphi (b) \in \Theta$.
\item[(3)]
Suppose that $(a, b \Omega) \not \in (\bar \Theta \times \bar \Theta)$.
There exists $\varepsilon \in \{0, 1\}$ such that $\dpt (c \Omega) = \dpt (a) + \dpt (b \Omega) - p(2k-1) - \varepsilon$.
Moreover, $\varepsilon = 1$ if either $a \in \Theta$ or $b \Omega \in \Theta$.
\item[(4)]
Suppose that $(a \Omega, \varphi (b)\, \Omega) \not \in (\bar \Theta \times \bar \Theta)$.
There exists $\varepsilon \in \{0, 1\}$ such that $\dpt (c) = \dpt (a \Omega) + \dpt (\varphi (b)\, \Omega) - (p+1)(2k-1) - \varepsilon$.
Moreover, $\varepsilon = 1$ if either $a \Omega \in \Theta$ or $\varphi (b)\, \Omega \in \Theta$ or $c \in M_1$.
\end{itemize}
\end{lem}

\begin{proof}
We write $a$ and $b$ in the form $a = a_{2p+1} a_{2p} \cdots a_2 a_1$ and $b = b_1 b_2 \cdots b_{2p} b_{2p+1}$ so that:
\begin{itemize}
\item
$a_i \neq 1$, $b_i \neq 1$, $a_i b_i = \Omega$ if $i$ is odd, and $a_i\,\varphi (b_i) = \Omega$ if $i$ is even, for all $i \in \{1, \dots, 2p\}$;
\item
$c = a_{2p+1} b_{2p+1}$;
\item
Set $x_i = \tau(a_i)$, $x_i' = \sigma (a_i)$, $y_i = \sigma (b_i)$, $y_i' = \tau (b_i)$, for all $i \in \{1, \dots, 2p+1\}$.
Then $x_{i+1} = x_i'$ for all $i \in \{1, \dots, 2p-1\}$.
\end{itemize}
We have $y_i = \varphi (x_i)$ and $y_i' = x_i'$ if $i$ is odd, and $y_i = x_i$ and $y_i' = \varphi (x_i')$ if $i$ is even, for all $i \in \{1, \dots, 2p\}$.
Thus, if $i$ is odd, then $y_{i+1} = x_{i+1} = x_i' = y_i'$, and if $i$ is even, then $y_{i+1} = \varphi (x_{i+1}) = \varphi (x_i') = y_i'$, for $i \in \{1, \dots, 2p-1\}$.

Let $u = |\{ i \in \{1, \dots, 2p\} \mid x_i'=s \}|$.
By using Lemma \ref{lem6_7} we show successively the following equalities.  
\[
\begin{array}{c}
\dpt (a) = \dpt (a_{2p+1}) + \sum_{i=1}^{2p} \dpt (a_i) - u + \varepsilon_{1,a}\,,\\
\dpt(a \Omega) = \dpt (a_{2p+1}) + \sum_{i=1}^{2p} \dpt(a_i) - u + k + \varepsilon_{2,a}\,,
\end{array}
\]
where $\varepsilon_{1,a}$ and $\varepsilon_{2,a}$ are as follows. 
If $p \ge 1$ and $a_{2p+1} \neq 1$, then: 
\[
\begin{array}{l}
\varepsilon_{1,a} = \left\{ \begin{array}{ll}
0 & \text{if } (x_{2p}', x_{2p+1}) \in \{(s,s), (t,s), (t,t) \}\,,\\
1 &\text{if } (x_{2p}', x_{2p+1}) = (s,t)\,,
\end{array}\right.\\
\varepsilon_{2,a} = \left\{ \begin{array}{ll}
-1 &\text{if } (x_1, x_{2p}', x_{2p+1}) \in \{ (s,s,s), (s,t,s), (s,t,t) \}\,,\\
0 &\text{if } (x_1, x_{2p}', x_{2p+1}) \in \{ (s,s,t), (t,s,s), (t,t,s), (t,t,t) \}\,,\\
1 &\text{if } (x_1, x_{2p}', x_{2p+1}) = (t,s,t)\,.
\end{array} \right.
\end{array}
\]
If $p \ge 1$ and $a_{2p+1}=1$, then:
\[
\varepsilon_{1,a} = \left\{ \begin{array}{ll}
0 &\text{if } x_{2p}' = t\,,\\
1 &\text{if } x_{2p}' = s\,,
\end{array} \right.\
\varepsilon_{2,a} = \left\{ \begin{array}{ll}
-1 &\text{if } (x_1, x_{2p}') = (s,t)\,,\\
0 &\text{if } (x_1, x_{2p}') \in \{ (s,s), (t,t) \}\,,\\
1 &\text{if } (x_1, x_{2p}') = (t,s)\,.
\end{array} \right.
\]
If $p=0$ and $a \neq 1$, then:
\[
\varepsilon_{1,a} = 0\,,\
\varepsilon_{2,a} = \left\{ \begin{array}{ll}
-1 &\text{if } x_{2p+1} = s\,,\\
0 &\text{if } x_{2p+1} = t\,.
\end{array} \right.
\]
If $p=0$ and $a =1$, then $\varepsilon_{1,a} = \varepsilon_{2,a} = 0$.

Let $v = |\{ i \in \{1, \dots, 2p\} \mid y_i'=s \}|$.
Similarly, by using Lemma \ref{lem6_7} we prove successively the following equalities.
\[
\begin{array}{c}
\dpt(b) = \dpt(b_{2p+1}) + \sum_{i=1}^{2p} \dpt(b_i) - v + \varepsilon_{1,b}\,,\\
\dpt (\varphi (b)) = \dpt(b_{2p+1}) + \sum_{i=1}^{2p} \dpt(b_i) - v + \varepsilon_{2,b}\,,\\
\dpt (b \Omega) = \dpt(b_{2p+1}) + \sum_{i=1}^{2p} \dpt(b_i) - v + k + \varepsilon_{3,b}\,,\\
\dpt(\varphi(b)\, \Omega) = \dpt(b_{2p+1}) + \sum_{i=1}^{2p} \dpt(b_i) - v + k + \varepsilon_{4,b}\,,
\end{array} 
\]
where $\varepsilon_{1,b}$, $\varepsilon_{2,b}$, $\varepsilon_{3,b}$ and $\varepsilon_{4,b}$ are as follows. 
If $p \ge 1$ and $b_{2p+1} \neq 1$, then:
\[
\begin{array}{l}
\varepsilon_{1,b} = \left\{\begin{array}{ll}
0 &\text{if } (y_{2p}', y_{2p+1}) \in \{ (s,s), (t,s), (t,t) \}\,,\\
1 &\text{if } (y_{2p}', y_{2p+1}) = (s,t)\,,
\end{array}\right.\\
\varepsilon_{2,b} = \left\{ \begin{array}{ll}
-1  &\text{if } (y_{1}, y_{2p}', y_{2p+1}, y_{2p+1}') \in \{ (s,s,s,s), (s,t,s,s), (s,t,t,s) \}\,,\\
0  &\text{if } (y_{1}, y_{2p}', y_{2p+1}, y_{2p+1}') \in \{ (s,s,s,t), (s,s,t,s), (s,t,s,t),\\
& (s,t,t,t), (t,s,s,s), (t,t,s,s), (t,t,t,s) \}\,,\\
1  &\text{if } (y_{1}, y_{2p}', y_{2p+1}, y_{2p+1}') \in \{ (s,s,t,t), (t,s,s,t), (t,s,t,s),\\
& (t,t,s,t), (t,t,t,t) \}\,,\\
2  &\text{if } (y_{1}, y_{2p}', y_{2p+1}, y_{2p+1}') = (t,s,t,t)\,,
\end{array} \right.\\
\varepsilon_{3,b} = \left\{ \begin{array}{ll}
-1  &\text{if } (y_{2p}',y_{2p+1},y_{2p+1}') \in \{(s,s,s), (t,s,s), (t,t,s) \}\,,\\
0  &\text{if } (y_{2p}',y_{2p+1},y_{2p+1}') \in \{ (s,s,t), (s,t,s), (t,s,t), (t,t,t) \}\,,\\
1  &\text{if } (y_{2p}',y_{2p+1},y_{2p+1}') = (s,t,t)\,,
\end{array} \right.\\
\varepsilon_{4,b} = \left\{ \begin{array}{ll}
-1  &\text{if } (y_{1},y_{2p}',y_{2p+1}) \in \{ (s,s,s), (s,t,s), (s,t,t) \}\,,\\
0  &\text{if } (y_{1},y_{2p}',y_{2p+1}) \in \{ (s,s,t), (t,s,s), (t,t,s), (t,t,t) \}\,,\\
1  &\text{if } (y_{1},y_{2p}',y_{2p+1}) = (t,s,t)\,.
\end{array} \right.
\end{array}
\]
If $p\ge 1$ and $b_{2p+1} = 1$, then:
\begin{gather*}
\varepsilon_{1,b} = \left\{ \begin{array}{ll}
0 &\text{if } y_{2p}' = t\,,\\
1 &\text{if } y_{2p}' = s\,,
\end{array} \right.\
\varepsilon_{2,b} = \left\{ \begin{array}{ll}
0 &\text{if } y_{1} = s\,,\\
1 &\text{if } y_{1} = t\,,
\end{array} \right. \\
\varepsilon_{3,b} = 0\,,\
\varepsilon_{4,b} = \left\{ \begin{array}{ll}
-1 &\text{if } (y_{1},y_{2p}') = (s,t)\,,\\
0 &\text{if } (y_{1},y_{2p}') \in \{ (s,s), (t,t) \}\,,\\
1 &\text{if } (y_{1},y_{2p}') = (t,s)\,.
\end{array} \right.
\end{gather*}
If $p=0$ and $b \neq 1$, then: 
\begin{gather*}
\varepsilon_{1,b} = 0\,,\
\varepsilon_{2,b} = \left\{ \begin{array}{ll}
-1 &\text{if } (y_{2p+1},y_{2p+1}') = (s,s)\,,\\
0 &\text{if } (y_{2p+1},y_{2p+1}') \in \{ (s,t), (t,s) \}\,,\\
1 &\text{if } (y_{2p+1},y_{2p+1}') = (t,t)\,,
\end{array}\right.\\
\varepsilon_{3,b}=\left\{ \begin{array}{ll}
-1 &\text{if } y_{2p+1}' = s\,,\\
0 &\text{if } y_{2p+1}' = t\,,
\end{array}\right.\
\varepsilon_{4,b} = \left\{ \begin{array}{ll}
-1 &\text{if } y_{2p+1} = s\,,\\
0 &\text{if } y_{2p+1} = t\,.
\end{array}\right.
\end{gather*}
If $p=0$ and $b=1$, then $\varepsilon_{1,b} = \varepsilon_{2,b} = \varepsilon_{3,b} = \varepsilon_{4,b} = 0$.

Again, by applying Lemma \ref{lem6_7} we prove successively the following equalities.
\[
\begin{array}{c}
\dpt(c) = \dpt(a_{2p+1}) + \dpt(b_{2p+1}) + \varepsilon_{1,c}\,,\\
\dpt(c \Omega) = \dpt(a_{2p+1}) + \dpt(b_{2p+1}) + k + \varepsilon_{2,c}\,,
\end{array}
\]
where $\varepsilon_{1,c}$ and $\varepsilon_{2,c}$ are as follows. 
If $a_{2p+1} \neq 1$ and $b_{2p+1} \neq 1$, then:
\[
\begin{array}{l}
\varepsilon_{1,c} = \left\{ \begin{array}{ll}
-1 &\text{if } (x_{2p+1},y_{2p+1}) = (s,s)\,,\\
0 &\text{if } (x_{2p+1},y_{2p+1}) \in \{ (s,t), (t,s), (t,t) \}\,,
\end{array} \right.\\
\varepsilon_{2,c} = \left\{ \begin{array}{ll}
-2 &\text{if } (x_{2p+1}, y_{2p+1}, y_{2p+1}') = (s,s,s)\,,\\
-1 &\text{if } (x_{2p+1}, y_{2p+1}, y_{2p+1}') \in \{ (s,s,t), (s,t,s), (t,s,s), (t,t,s) \}\,,\\
0 &\text{if } (x_{2p+1}, y_{2p+1}, y_{2p+1}') \in \{ (s,t,t), (t,s,t), (t,t,t) \}\,.
\end{array} \right.
\end{array}
\]
If $a_{2p+1} \neq 1$ and $b_{2p+1} = 1$, then:
\[
\varepsilon_{1,c} = 0\,,\
\varepsilon_{2,c} = \left\{ \begin{array}{ll}
-1 &\text{if } x_{2p+1} = s\,,\\
0 &\text{if } x_{2p+1} = t\,.
\end{array} \right.
\]
If $a_{2p+1} = 1$ and $b_{2p+1} \neq 1$, then:
\[
\varepsilon_{1,c} = 0\,,\
\varepsilon_{2,c} = \left\{ \begin{array}{ll}
-1 &\text{if } y_{2p+1}' = s\,,\\
0 &\text{if } y_{2p+1}' = t\,.
\end{array} \right.
\]
If $a_{2p+1} = 1$ and $b_{2p+1} = 1$, then $\varepsilon_{1,c} = \varepsilon_{2,c} = 0$.

From Lemma \ref{lem6_8} we also get $\sum_{i=1}^{2p} ( \dpt(a_i) + \dpt (b_i) ) = p(2k-1) + u + v$.

{\it Part (1):}
Let $\varepsilon = \varepsilon_{1,a} + \varepsilon_{1,b} - \varepsilon_{1,c}$.
By the above we have $\dpt(c) = \dpt(a) + \dpt(b) - p(2k-1) - \varepsilon$, and $\varepsilon$ is as follows.
If $p \ge 1$, $a_{2p+1} \neq 1$ and $b_{2p+1} \neq 1$, then: $\varepsilon = 0$ if $(x_{2p}', x_{2p+1}, y_{2p+1}) \in \{ (s,s,t), (t,t,s) \}$, and $\varepsilon = 1$ otherwise. 
If $p \ge 1$, $a_{2p+1} \neq 1$ and $b_{2p+1} = 1$, then: $\varepsilon = 0$ if $(x_{2p}', x_{2p+1}) = (s,s)$, and $\varepsilon = 1$ otherwise. 
If $p \ge 1$, $a_{2p+1} = 1$ and $b_{2p+1} \neq 1$, then: $\varepsilon = 0$ if $(x_{2p}', y_{2p+1}) = (t,s)$, and $\varepsilon = 1$ otherwise. 
If $p \ge 1$, $a_{2p+1} = 1$ and $b_{2p+1} = 1$, then $\varepsilon = 1$.
If $p=0$, $a=a_{2p+1} \neq 1$ and $b= b_{2p+1} \neq 1$, then: $\varepsilon = 0$ if $(x_{2p+1},y_{2p+1}) \in \{ (s,t), (t,s), (t,t) \}$, and $\varepsilon = 1$ otherwise. 
If $p=0$ and $a=a_{2p+1} = 1$, then $\varepsilon = 0$.
If $p=0$ and $b= b_{2p+1} = 1$, then $\varepsilon = 0$.

Suppose that $a \in \Theta$.
Then $a$ is written $a = \theta^q$ with $q \ge 1$.
On the other hand we write $b= t^r b'$ where $b' \neq 1$ (since $b \not\in \bar \Theta$) and $\sigma(b') = s$.
If $r=0$, then $p=0$, $x_{2p+1}=s$ and $y_{2p+1}=s$, hence $\varepsilon = 1$.
If $1 \le r < 2q$, then $r=2p$, $a_{2p+1} = \theta^{q-p}$, $b_{2p+1} = b'$ and $(x_{2p}', x_{2p+1}, y_{2p+1}) = (s,s,s)$, hence $\varepsilon = 1$.
If $r \ge 2q$, then $q=p$, $a_{2p+1}=1$, $b_{2p+1}\neq 1$ and $x_{2p}'=s$, hence $\varepsilon = 1$.
The case $b \in \Theta$ is proved in the same way.

Suppose that $c \in M_1$.
Then $p \ge 1$, since $(a,b) \not\in \bar \Theta \times \bar \Theta$.
If $a_{2p+1} \neq 1$ and $b_{2p+1} \neq 1$, then $(x_{2p+1}, y_{2p+1}) = (t,t)$, hence $\varepsilon = 1$.
If $a_{2p+1} \neq 1$ and $b_{2p+1} = 1$, then $x_{2p+1} = t$, hence $\varepsilon = 1$.
If $a_{2p+1} = 1$ and $b_{2p+1} \neq 1$, then $y_{2p+1} = t$, hence $\varepsilon = 1$.
If $a_{2p+1} = 1$ and $b_{2p+1} = 1$, then $\varepsilon = 1$.

{\it Part (2):}
Let $\varepsilon = \varepsilon_{2,a} + \varepsilon_{2,b} - \varepsilon_{2,c}$.
By the above we have $\dpt (c \Omega) = \dpt (a \Omega) + \dpt (\varphi(b)) - p(2k-1) - \varepsilon$, and $\varepsilon$ is as follows. 
If $p \ge 1$, $a_{2p+1} \neq 1$ and $b_{2p+1} \neq 1$, then: $\varepsilon=0$ if $(x_{2p}',x_{2p+1}, y_{2p+1}) \in \{ (s,s,t), (t,t,s) \}$, and $\varepsilon = 1$ otherwise.
If $p \ge 1$, $a_{2p+1} \neq 1$ and $b_{2p+1} = 1$, then: $\varepsilon = 0$ if $(x_1, x_{2p}', x_{2p+1}) \in \{ (t,s,s), (t,t,s), (t,t,t) \}$, and $\varepsilon = 1$ otherwise.
If $p \ge 1$, $a_{2p+1} = 1$ and $b_{2p+1} \neq 1$, then: $\varepsilon = 0$ if $(x_{2p}',y_{2p+1}) = (t,s)$, and $\varepsilon = 1$ otherwise.
If $p \ge 1$, $a_{2p+1} = 1$ and $b_{2p+1} = 1$, then: $\varepsilon = 0$ if $x_{2p}' = t$, and $\varepsilon = 1$ otherwise.
If $p=0$, $a=a_{2p+1} \neq 1$ and $b=b_{2p+1} \neq 1$, then: $\varepsilon = 0$ if $(x_{2p+1}, y_{2p+1}) \in \{ (s,s), (s,t), (t,s) \}$, and $\varepsilon = 1$ otherwise.
If $p=0$, $a=a_{2p+1} = 1$ and $b=b_{2p+1} \neq 1$, then: $\varepsilon=0$ if $y_{2p+1}=s$, and $\varepsilon=1$ otherwise.
If $p=0$ and $b=b_{2p+1} = 1$, then $\varepsilon=0$.

Suppose that $a \Omega \in \Theta$.
Then $a \Omega$ is written $a \Omega = \theta^q t$ with $q \ge 1$, hence $a = \theta^{q-1} (st)^k$.
On the other hand we write $b = s^r b'$, where $b' \neq 1$ (since $\varphi (b') \not \in \bar \Theta$) and $\sigma (b') = t$.
We necessarily have $r=2p \le 2(q-1)$, hence $a_{2p+1} = \theta^{q-p-1} (st)^k$ and $b_{2p+1} = b'$.
If $p \ge 1$, then $(x_{2p}', x_{2p+1}, y_{2p+1}) = (t,t,t)$, hence $\varepsilon = 1$.
If $p = 0$, then $(x_{2p+1}, y_{2p+1}) = (t,t)$, hence $\varepsilon = 1$.

Suppose that $\varphi (b) \in \Theta$.
Then $\varphi (b)$ is written $\varphi (b) = \theta^q$ with $q \ge 1$, hence $b = ((ts)^k (st)^k)^q$.
On the other hand we write $a = a' s^r$ where either $a'=1$ or $\tau(a') = t$.
If $r = 0$ and $a'=1$, then $p=0$, $b_{2p+1} = b \neq 1$ and $y_{2p+1} = t$, hence $\varepsilon = 1$.
If $r = 0$ and $a' \neq 1$, then $p=0$, $a_{2p+1} = a' \neq 1$, $b_{2p+1} = b \neq 1$ and $(x_{2p+1}, y_{2p+1}) = (t,t)$, hence $\varepsilon = 1$.
If $0 < r < 2q$ and $a'=1$, then $r=2p$, $a_{2p+1} = 1$, $b_{2p+1} = ((ts)^k (st)^k)^{q-p} \neq 1$ and $(x_{2p}',y_{2p+1}) = (s,t)$, hence $\varepsilon = 1$.
If $0 < r < 2q$ and $a' \neq 1$, then $r=2p$, $a_{2p+1} = a' \neq 1$, $b_{2p+1} = ((ts)^k (st)^k)^{q-p} \neq 1$ and $(x_{2p}', x_{2p+1}, y_{2p+1}) = (s,t,t)$, hence $\varepsilon = 1$.
If $r=2q$ and $a' = 1$, then $r =2p$, $a_{2p+1}=1$, $b_{2p+1}=1$ and $x_{2p}'=s$, hence $\varepsilon=1$.
If $r=2q$ and $a' \neq 1$, then $r =2p$, $a_{2p+1} = a' \neq 1$, $b_{2p+1}=1$ and $(x_1, x_{2p}', x_{2p+1}) = (s,s,t)$, hence $\varepsilon = 1$.
If $r>2q$, then $p=q$, $a_{2p+1} = a' s^{r-2q} \neq 1$, $b_{2p+1} = 1$, and $(x_1, x_{2p}', x_{2p+1}) = (s,s,s)$, hence $\varepsilon = 1$.

{\it Part (3):}
Let $\varepsilon = \varepsilon_{1,a} + \varepsilon_{3,b} - \varepsilon_{2,c}$.
By the above we have $\dpt (c \Omega) = \dpt(a) + \dpt (b \Omega) - p(2k-1) - \varepsilon$, and $\varepsilon$ is as follows.
If $p \ge 1$,  $a_{2p+1} \neq 1$ and $b_{2p+1} \neq 1$, then: $\varepsilon = 0$ if $(x_{2p}',x_{2p+1},y_{2p+1}) \in \{ (s,s,t), (t,t,s) \}$, and $\varepsilon = 1$ otherwise. 
If $p \ge 1$, $a_{2p+1} \neq 1$ and $b_{2p+1} = 1$, then: $\varepsilon = 0$ if $(x_{2p}', x_{2p+1}) = (t,t)$, and $\varepsilon = 1$ otherwise. 
If $p \ge 1$, $a_{2p+1} = 1$ and $b_{2p+1} \neq 1$, then: $\varepsilon = 0$ if $(x_{2p}',y_{2p+1}) = (t,s)$, and $\varepsilon = 1$ otherwise. 
If $p \ge 1$, $a_{2p+1} = 1$ and $b_{2p+1} = 1$, then: $\varepsilon = 0$ if $x_{2p}' = t$, and $\varepsilon = 1$ otherwise. 
If $p = 0$, $a \neq 1$ and $b \neq 1$, then: $\varepsilon=0$ if $(x_{2p+1},y_{2p+1}) \in \{ (s,t), (t,s), (t,t) \}$, and $\varepsilon= 1$ otherwise. 
If $p = 0$, $a \neq 1$ and $b = 1$, then: $\varepsilon = 0$ if $x_{2p+1} = t$, and $\varepsilon = 1$ otherwise. 
If $p = 0$ and $a = 1$, then $\varepsilon = 0$.

Suppose that $a \in \Theta$.
Then $a$ is written $a = \theta^q$ with $q \ge 1$.
On the other hand we write $b = t^r b'$ where either $b'=1$ or $\sigma(b') = s$.
If $r=0$ and $b'=1$, then $p=0$, $a_{2p+1} = \theta^q$, $b_{2p+1}=1$ and $x_{2p+1}=s$, hence $\varepsilon = 1$.
If $r=0$ and $b' \neq 1$, then $p=0$, $a_{2p+1}= \theta^q$, $b_{2p+1} = b' \neq 1$ and $(x_{2p+1},y_{2p+1}) = (s,s)$, hence $\varepsilon = 1$.
If $0 < r < 2q$ and $b'=1$, then $r=2p$, $a_{2p+1} = \theta^{q-p} \neq 1$, $b_{2p+1} = 1$ and $(x_{2p}', x_{2p+1}) =(s,s)$, hence $\varepsilon = 1$.
If $0 < r < 2q$ and $b'\neq 1$, then $r=2p$, $a_{2p+1} = \theta^{q-p} \neq 1$, $b_{2p+1} = b' \neq 1$ and $(x_{2p}', x_{2p+1}, y_{2p+1}) = (s,s,s)$, hence $\varepsilon = 1$.
If $r=2q$ and $b'=1$, then $p=q$, $a_{2p+1} = 1$, $b_{2p+1} = 1$ and $x_{2p}' = s$, hence $\varepsilon = 1$.
If $r=2q$ and $b'\neq 1$, then $p=q$, $a_{2p+1} = 1$, $b_{2p+1} = b'\neq 1$ and $(x_{2p}',y_{2p+1}) = (s,s)$, hence $\varepsilon = 1$.
If $r>2q$, then $a_{2p+1}=1$, $b_{2p+1} = t^{r-2q} b' \neq 1$ and $(x_{2p}',y_{2p+1}) = (s,t)$, hence $\varepsilon = 1$.

Suppose that $b \Omega \in \Theta$.
Then $b \Omega$ is written $b \Omega = \theta^q t$ with $q \ge 1$, hence $b = \theta^{q-1} (st)^k$.
On the other hand we write $a = a' t^r$ where $a' \neq 1$ (since $a \not\in \bar\Theta$) and $\tau (a') = s$.
If $r=0$, then $p=0$, $a_{2p+1}=a' \neq 1$, $b_{2p+1} = \theta^{q-1} (st)^k$ and $(x_{2p+1},y_{2p+1}) = (s,s)$, hence $\varepsilon = 1$.
If $r > 0$, then $r=2p \le 2(q-1)$, $a_{2p+1} = a' \neq 1$, $b_{2p+1} = \theta^{q-p-1} (st)^k$ and $(x_{2p}', x_{2p+1}, y_{2p+1}) = (t,s,s)$, hence $\varepsilon = 1$.

{\it Part (4):}
Let $\varepsilon = 1 + \varepsilon_{2,a} + \varepsilon_{4,b} - \varepsilon_{1,c}$.
By the above we have $\dpt (c) = \dpt (a \Omega) + \dpt (\varphi(b) \Omega) - (p+1)(2k-1) - \varepsilon$, and $\varepsilon$ is as follows. 
If $p \ge 1$, $a_{2p+1} \neq 1$ and $b_{2p+1} \neq 1$, then: $\varepsilon=0$ if $(x_{2p}', x_{2p+1}, y_{2p+1}) \in \{(s,s,t), (t,t,s) \}$, and $\varepsilon=1$ otherwise.
If $p \ge 1$, $a_{2p+1} \neq 1$ and $b_{2p+1} = 1$, then: $\varepsilon=0$ if $(x_{2p}',x_{2p+1}) = (s,s)$, and $\varepsilon=1$ otherwise. 
If $p \ge 1$, $a_{2p+1} = 1$ and $b_{2p+1} \neq 1$, then: $\varepsilon=0$ if $(x_{2p}', y_{2p+1}) = (t,s)$, and $\varepsilon=1$ otherwise.
If $p \ge 1$, $a_{2p+1} = 1$ and $b_{2p+1} = 1$, then $\varepsilon = 1$.
If $p=0$, $a \neq 1$ and $b \neq 1$, then: $\varepsilon= 0$ if $(x_{2p+1},y_{2p+1}) \in \{ (s,s) , (s,t), (t,s) \}$, and $\varepsilon= 1$ otherwise. 
If $p=0$, $a \neq 1$ and $b = 1$, then: $\varepsilon= 0$ if $x_{2p+1} = s$, and $\varepsilon= 1$ otherwise. 
If $p=0$, $a = 1$ and $b \neq 1$, then: $\varepsilon = 0$ if $y_{2p+1} = s$, and $\varepsilon = 1$ otherwise. 
If $p=0$, $a = 1$ and $b = 1$, then $\varepsilon = 1$.

Suppose that $a \Omega \in \Theta$.
Then $a \Omega$ is written $a \Omega = \theta^q t$ with $q \ge 1$, hence $a = \theta^{q-1} (st)^k$.
On the other hand we write $b = s^r b'$, where either $b'=1$ or $\sigma (b') = t$.
If $r=0$ and $b'=1$, then $p=0$, $a=\theta^{q-1} (st)^k \neq 1$, $b=1$ and $x_{2p+1}=t$, hence $\varepsilon = 1$.
If $r=0$ and $b'\neq 1$, then $p=0$, $a=\theta^{q-1} (st)^k \neq 1$, $b=b' \neq 1$ and $(x_{2p+1},y_{2p+1}) = (t,t)$, hence $\varepsilon = 1$.
If $r>0$ and $b'=1$, then $r=2p \le 2(q-1)$, $a_{2p+1}=\theta^{q-p-1} (st)^k \neq 1$, $b_{2p+1} = b' = 1$ and $(x_{2p}',x_{2p+1}) = (t,t)$, hence $\varepsilon = 1$.
If $r>0$ and $b' \neq 1$, then $r = 2p \le 2(q-1)$, $a_{2p+1} = \theta^{q-p-1} (st)^k \neq 1$, $b_{2p+1} = b' \neq 1$ and $(x_{2p}', x_{2p+1}, y_{2p+1}) = (t,t,t)$, hence $\varepsilon = 1$.

Suppose that $\varphi(b) \Omega \in \Theta$.
Then $\varphi(b) \Omega$ is written $\varphi (b) \Omega = \theta^q t$ with $q \ge 1$, hence $b = (ts)^k \theta^{q-1}$.
On the other hand we write $a = a' s^r$ where either $a'=1$ or $\tau (a') = t$.
If $r=0$ and $a'=1$, then $p=0$, $a=1$, $b=(ts)^k \theta^{q-1} \neq 1$ and $y_{2p+1}=t$, hence $\varepsilon = 1$.
If $r=0$ and $a' \neq 1$, then $p=0$, $a=a' \neq 1$, $b=(ts)^k \theta^{q-1} \neq 1$ and $(x_{2p+1},y_{2p+1}) = (t,t)$, hence $\varepsilon = 1$.
If $r > 0$ and $a' = 1$, then $r = 2p \le 2(q-1)$, $a_{2p+1}=1$, $b_{2p+1} = (ts)^k \theta^{q-p-1} \neq 1$ and $(x_{2p}', y_{2p+1}) = (s,t)$, hence $\varepsilon = 1$.
If $r>0$ and $a' \neq 1$, then $r = 2p \le 2(q-1)$, $a_{2p+1} = a' \neq 1$, $b_{2p+1} = (ts)^k \theta^{q-p-1} \neq 1$ and $(x_{2p}', x_{2p+1}, y_{2p+1}) = (s,t,t)$, hence $\varepsilon = 1$.

Suppose that $c \in M_1$.
If $p \ge 1$, $a_{2p+1} \neq 1$ and $b_{2p+1} \neq 1$, then $(x_{2p+1}, y_{2p+1}) = (t,t)$, hence $\varepsilon=1$.
If $p \ge 1$, $a_{2p+1} \neq 1$ and $b_{2p+1} = 1$, then $x_{2p+1}=t$, hence $\varepsilon=1$.
If $p \ge 1$, $a_{2p+1} = 1$ and $b_{2p+1} \neq 1$, then $y_{2p+1}=t$, hence $\varepsilon=1$.
If $p \ge 1$, $a_{2p+1} = 1$ and $b_{2p+1} = 1$, then $\varepsilon=1$.
If $p=0$, $a \neq 1$ and $b \neq 1$, then $(x_{2p+1},y_{2p+1}) = (t,t)$, hence $\varepsilon=1$.
If $p=0$, $a \neq 1$ and $b = 1$, then $x_{2p+1} = t$, hence $\varepsilon=1$.
If $p=0$, $a = 1$ and $b \neq 1$, then $y_{2p+1} = t$, hence $\varepsilon=1$.
If $p=0$, $a = 1$ and $b = 1$, then $\varepsilon=1$.
\end{proof}

\begin{lem}\label{lem6_10}
Let $a,b$ be two $\Omega$-unmovable elements of $M$.
We assume that the $\Delta$-form of $ab$ is in the form $ab = (c \Omega) \Delta^p$ where $c$ is an $\Omega$-unmovable element of $M$ and $p \ge 0$.
\begin{itemize}
\item[(1)]
Suppose that $(a,b) \not\in (\bar \Theta \times \bar \Theta)$.
There exists $\varepsilon \in \{0, 1\}$ such that $\dpt (c \Omega) = \dpt(a) + \dpt (b) - p(2k-1) - \varepsilon$.
Moreover, $\varepsilon = 1$ if either $a \in \Theta$ or $b \in \Theta$.
\item[(2)]
Suppose that $( a \Omega, \varphi (b)) \not \in (\bar \Theta \times \bar \Theta)$.
There exists $\varepsilon \in \{0, 1\}$ such that $\dpt (c) = \dpt (a \Omega) + \dpt (\varphi (b)) - (p+1) (2k-1) - \varepsilon$.
Moreover, $\varepsilon =1$ if either $a \Omega \in \Theta$ or $\varphi (b) \in \Theta$ or $c \in M_1$.
\item[(3)]
Suppose that $(a, b \Omega) \not \in (\bar \Theta \times \bar \Theta)$.
There exists $\varepsilon \in \{0, 1 \}$ such that $\dpt(c) = \dpt(a) + \dpt (b \Omega) - (p+1) (2k-1) - \varepsilon$.
Moreover, $\varepsilon = 1$ if either $a \in \Theta$ or $b \Omega \in \Theta$ or $c \in M_1$.
\item[(4)]
Suppose that $(a \Omega, \varphi (b) \Omega) \not \in (\bar \Theta \times \bar \Theta)$.
There exists $\varepsilon \in \{0, 1\}$ such that $\dpt (c \Omega) = \dpt (a \Omega) + \dpt (\varphi (b)\, \Omega) - (p+1)(2k-1) - \varepsilon$.
Moreover, $\varepsilon = 1$ if either $a \Omega \in \Theta$ or $\varphi (b)\, \Omega \in \Theta$.
\end{itemize}
\end{lem}

\begin{proof}
We write $a$ and $b$ in the form $a = a_{2p+2} a_{2p+1} \cdots a_2 a_1$ and $b = b_1 b_2 \cdots b_{2p+1}
\allowbreak
b_{2p+2}$ so that:
\begin{itemize}
\item
$a_i \neq 1$, $b_i \neq 1$, $a_i b_i = \Omega$ if $i$ is odd, and $a_i\, \varphi (b_i) = \Omega$ if $i$ is even, for all $i \in \{1, \dots, 2p+1\}$;
\item
$c = a_{2p+2}\, \varphi (b_{2p+2})$.
\item
Set $x_i = \tau (a_i)$, $x_i' = \sigma (a_i)$, $y_i = \sigma (b_i)$ and $y_i' = \tau (b_i)$ for all $i \in \{1, \dots, 2p+2\}$.
Then $x_{i+1} = x_i'$ for all $i \in \{1, \dots, 2p\}$.
\end{itemize}
For $i \in \{1, \dots, 2p+1\}$ we have $y_i = \varphi (x_i)$ and $y_i' = x_i'$ if $i$ is odd and $y_i = x_i$ and $y_i' = \varphi (x_i')$ if $i$ is even.
So, if $i \in \{1, \dots, 2p\}$, then $y_{i+1} = x_{i+1} = x_i' = y_i'$ if $i$ is odd, and $y_{i+1} = \varphi(x_{i+1}) = \varphi (x_i') = y_i'$ if $i$ is even. 

Let $u = | \{ i \in \{1, \dots, 2p+1\} \mid x_i' = s \}|$.
By using Lemma \ref{lem6_7} we obtain successively the following equalities. 
\[
\begin{array}{c}
\dpt(a)  =  \dpt (a_{2p+2}) + \sum_{i=1}^{2p+1} \dpt(a_i) - u + \varepsilon_{1,a}\,,\\
\dpt(a\Omega) =  \dpt (a_{2p+2}) + \sum_{i=1}^{2p+1} \dpt(a_i) - u + k + \varepsilon_{2,a}\,,
\end{array}
\]
where $\varepsilon_{1,a}$ and $\varepsilon_{2,a}$ are as follows. 
If $a_{2p+2} \neq 1$, then:
\[
\begin{array}{l}
\varepsilon_{1,a} = \left\{ \begin{array}{ll}
0 & \text{if } (x_{2p+1}',x_{2p+2}) \in \{ (s,s), (t,s), (t,t) \}\,,\\ 
1 & \text{if }(x_{2p+1}',x_{2p+2}) = (s,t)\,,
\end{array}\right.\\
\varepsilon_{2,a} = \left\{
\begin{array}{ll}
-1 & \text{if } (x_1,x_{2p+1}',x_{2p+2}) \in \{ (s,s,s), (s,t,s), (s,t,t) \}\,,\\ 
0 & \text{if }(x_1,x_{2p+1}',x_{2p+2}) \in \{ (s,s,t), (t,s,s), (t,t,s), (t,t,t) \}\,,\\ 
1 & \text{if } (x_1,x_{2p+1}',x_{2p+2}) = (t,s,t)\,.
\end{array} \right. \end{array}
\]
If $a_{2p+2}=1$, then:
\[
\varepsilon_{1,a} = \left\{ \begin{array}{ll}
0 & \text{if } x_{2p+1}' = t\,,\\
1 & \text{if }x_{2p+1}' = s\,,
\end{array}\right.\
\varepsilon_{2,a} = \left\{ \begin{array}{ll}
-1 & \text{if }(x_1,x_{2p+1}') = (s,t)\,,\\ 
0 & \text{if } (x_1,x_{2p+1}') \in \{ (s,s), (t,t) \}\,,\\
1 & \text{if } (x_1,x_{2p+1}') = (t,s)\,.
\end{array}\right.
\]

Let $v = |\{i \in \{1, \dots, 2p+1\} \mid y_i'= s \}|$.
Similarly, by using Lemma \ref{lem6_7} we obtain successively the following equalities.
\[
\begin{array}{c}
\dpt(b) = \dpt(b_{2p+2}) + \sum_{i=1}^{2p+1} \dpt (b_i) - v + \varepsilon_{1,b}\,,\\
\dpt(\varphi (b)) = \dpt(b_{2p+2}) + \sum_{i=1}^{2p+1} \dpt (b_i) - v + \varepsilon_{2,b}\,,\\
\dpt(b \Omega) = \dpt(b_{2p+2}) + \sum_{i=1}^{2p+1} \dpt (b_i) - v + k + \varepsilon_{3,b}\,,\\
\dpt(\varphi (b)\,\Omega) = \dpt(b_{2p+2}) + \sum_{i=1}^{2p+1} \dpt (b_i) - v + k + \varepsilon_{4,b}\,,
\end{array}
\]
where $\varepsilon_{1,b}$, $\varepsilon_{2,b}$, $\varepsilon_{3,b}$ and $\varepsilon_{4,b}$ are as follows. 
If $b_{2p+2} \neq 1$, then:
\[
\begin{array}{l}
\varepsilon_{1,b} = \left\{ \begin{array}{ll}
0 & \text{if } (y_{2p+1}',y_{2p+2}) \in \{ (s,s), (t,s), (t,t) \}\,,\\
1 & \text{if } (y_{2p+1}',y_{2p+2}) = (s,t)\,,
\end{array} \right.\\
\varepsilon_{2,b} = \left\{ \begin{array}{ll}
-1 & \text{if } (y_1, y_{2p+1}', y_{2p+2}, y_{2p+2}') \in \{(s,s,s,s), (s,t,s,s), (s,t,t,s) \}\,,\\ 
0 & \text{if } (y_1, y_{2p+1}', y_{2p+2}, y_{2p+2}') \in \{ (s,s,s,t), (s,s,t,s), (s,t,s,t),\\
& (s,t,t,t), (t,s,s,s), (t,t,s,s), (t,t,t,s) \}\,,\\
1 & \text{if } (y_1, y_{2p+1}', y_{2p+2}, y_{2p+2}') \in \{ (s,s,t,t), (t,s,s,t), (t,s,t,s),\\
& (t,t,s,t), (t,t,t,t) \}\,,\\
2  & \text{if } (y_1, y_{2p+1}',y_{2p+2}, y_{2p+2}') = (t,s,t,t)\,,
\end{array} \right.\\
\varepsilon_{3,b} = \left\{ \begin{array}{ll}
-1 & \text{if } (y_{2p+1}', y_{2p+2}, y_{2p+2}') \in \{(s,s,s), (t,s,s), (t,t,s) \}\,,\\ 
0 & \text{if } (y_{2p+1}', y_{2p+2}, y_{2p+2}') \in \{ (s,s,t), (s,t,s), (t,s,t), (t,t,t) \}\,,\\ 
1 & \text{if } (y_{2p+1}', y_{2p+2}, y_{2p+2}') = (s,t,t)\,,
\end{array}\right.\\
\varepsilon_{4,b} = \left\{ \begin{array}{ll}
-1 & \text{if } (y_1, y_{2p+1}', y_{2p+2}) \in \{ (s,s,s), (s,t,s), (s,t,t) \}\,,\\ 
0 & \text{if } (y_1, y_{2p+1}', y_{2p+2}) \in \{ (s,s,t), (t,s,s), (t,t,s), (t,t,t) \}\,,\\ 
1 & \text{if } (y_1, y_{2p+1}', y_{2p+2}) = (t,s,t)\,.
\end{array} \right. \end{array}
\]
If $b_{2p+2}=1$, then:
\begin{gather*}
\varepsilon_{1,b} = \left\{ \begin{array}{ll}
0 & \text{if } y_{2p+1}' = t\,,\\
1 & \text{if } y_{2p+1}' = s\,,
\end{array} \right.\
\varepsilon_{2,b} = \left\{ \begin{array}{ll}
0 & \text{if } y_1 = s\,,\\ 
1 & \text{if } y_1 = t\,,
\end{array}\right.\\
\varepsilon_{3,b} = 0\,,\
\varepsilon_{4,b}= \left\{ \begin{array}{ll}
-1 & \text{if } (y_1,y_{2p+1}') = (s,t)\,,\\ 
0 & \text{if } (y_1, y_{2p+1}') \in \{ (s,s), (t,t) \}\,,\\ 
1 & \text{if } (y_1, y_{2p+1}') = (t,s)\,.
\end{array} \right.
\end{gather*}

Again, by using Lemma \ref{lem6_7} we obtain the following equalities.
\[
\begin{array}{c}
\dpt(c) = \dpt(a_{2p+2}) + \dpt(b_{2p+2}) + \varepsilon_{1,c}\,,\\
\dpt(c \Omega) = \dpt(a_{2p+2}) + \dpt(b_{2p+2}) + k + \varepsilon_{2,c}\,,
\end{array}
\]
where $\varepsilon_{1,c}$ and $\varepsilon_{2,c}$ are as follows. 
If $a_{2p+2}\neq 1$ and $b_{2p+2}\neq 1$, then: 
\[
\begin{array}{l}
\varepsilon_{1,c} = \left\{ \begin{array}{ll}
-1 & \text{if } (x_{2p+2}, y_{2p+2}, y_{2p+2}') \in \{ (s,s,s), (s,t,s), (t,s,s) \}\,,\\ 
0 &\text{if } (x_{2p+2}, y_{2p+2}, y_{2p+2}') \in \{ (s,s,t), (s,t,t), (t,s,t), (t,t,s) \}\,,\\ 
1 &\text{if } (x_{2p+2}, y_{2p+2}, y_{2p+2}') = (t,t,t)\,,
\end{array} \right.\\
\varepsilon_{2,c} = \left\{ \begin{array}{ll}
-1 & \text{if } (x_{2p+2},y_{2p+2}) \in \{ (s,s), (s,t), (t,s) \}\,,\\ 
0 & \text{if } (x_{2p+2},y_{2p+2}) = (t,t)\,.
\end{array} \right. \end{array}
\]
If $a_{2p+2}\neq 1$ and $b_{2p+2}=1$, then:
\[
\varepsilon_{1,c} = 0\,, \
\varepsilon_{2,c} =\left\{ \begin{array}{ll}
-1 & \text{if } x_{2p+2}=s\,,\\ 
0 &\text{if } x_{2p+2}=t\,.
\end{array} \right.
\]
If $a_{2p+2}=1$ and $b_{2p+2}\neq 1$, then:
\[
\varepsilon_{1,c} = \left\{ \begin{array}{ll}
-1 &\text{if } (y_{2p+2},y_{2p+2}') = (s,s)\,,\\
0 &\text{if } (y_{2p+2},y_{2p+2}') \in \{ (s,t), (t,s) \}\,,\\ 
1 & \text{if } (y_{2p+2},y_{2p+2}') = (t,t)\,,
\end{array}\right.\
\varepsilon_{2,c} = \left\{ \begin{array}{ll}
-1 &\text{if } y_{2p+2} = s\,,\\ 
0 &\text{if } y_{2p+2} = t\,.
\end{array} \right.
\]
If $a_{2p+2}=1$ and $b_{2p+2}=1$, then $\varepsilon_{1,c} = \varepsilon_{2,c} = 0$.

Finally, from Lemma \ref{lem6_7} and Lemma \ref{lem6_8} follows that
\[
\sum_{i=1}^{2p+1} ( \dpt(a_i) + \dpt (b_i)) = p(2k-1) + k + u + v + \varepsilon_d\,,
\]
where $\varepsilon_d = -1$ if $x_{2p+1}' = s$, and $\varepsilon_d = 0$ if $x_{2p+1}' = t$.

{\it Part (1):}
Let $\varepsilon = \varepsilon_{1,a} + \varepsilon_{1,b} - \varepsilon_{2,c} + \varepsilon_d$.
By the above we have $\dpt(c \Omega) = \dpt(a) + \dpt(b) -p(2k-1) - \varepsilon$, where $\varepsilon$ is as follows. 
If $a_{2p+2} \neq 1$ and $b_{2p+2} \neq 1$, then $\varepsilon = 0$ if $(x_{2p+1}', x_{2p+2}, y_{2p+2}) \in \{(s,s,s), (t,t,t) \}$, and $\varepsilon = 1$ otherwise.
If $a_{2p+2} \neq 1$ and $b_{2p+2} = 1$, then $\varepsilon = 0$ if $(x_{2p+1}',x_{2p+2}) = (t,t)$, and $\varepsilon = 1$ otherwise. 
If $a_{2p+2} = 1$ and $b_{2p+2} \neq 1$, then $\varepsilon = 0$ if $(x_{2p+1}',y_{2p+2}) = (t,t)$, and $\varepsilon = 1$ otherwise. 
If $a_{2p+2} = 1$ and $b_{2p+2} = 1$, then $\varepsilon = 0$ if $x_{2p+1}' = t$, and $\varepsilon = 1$ otherwise.

Suppose that $a \in \Theta$. 
Then $a$ is written $a = \theta^q$ with $q \ge 1$. 
On the other hand we write $b = t^r b'$ where $b' \neq 1$ (since $b \not\in \bar \Theta$) and $\sigma (b') = s$.  
We necessarily have $r=2p+1<2q$, $a_{2p+2} = \theta^{q-p-1} (st)^k$, $b_{2p+2}=b'$, and $(x_{2p+1}', x_{2p+2}, y_{2p+2}) = (t,t,s)$, hence $\varepsilon = 1$.
The case $b \in \Theta$ is proved in a similar way.

{\it Part (2):}
Let $\varepsilon = 1 + \varepsilon_{2,a} + \varepsilon_{2,b} - \varepsilon_{1,c} + \varepsilon_d$.
By the above we have $\dpt(c) = \dpt (a \Omega) + \dpt( \varphi (b)) -(p+1)(2k-1) - \varepsilon$, and $\varepsilon$ is as follows. 
If $a_{2p+2} \neq 1$ and $b_{2p+2} \neq 1$, then $\varepsilon = 0$ if $(x_{2p+1}', x_{2p+2}, y_{2p+2}) \in \{ (s,s,s), (t,t,t) \}$, and $\varepsilon = 1$ otherwise. 
If $a_{2p+2} \neq 1$ and $b_{2p+2}=1$, then $\varepsilon = 0$ if $(x_{2p+1}', x_{2p+2}) = (s,s)$, and $\varepsilon = 1$ otherwise. If $a_{2p+2}=1$ and $b_{2p+2} \neq 1$, then $\varepsilon = 0$ if $(x_{2p+1}', y_{2p+2}) = (t,t)$, and $\varepsilon = 1$ otherwise. 
If $a_{2p+2}=1$ and $b_{2p+2} = 1$, then $\varepsilon = 1$.

Suppose that $a \Omega \in \Theta$.
Then $a\Omega$ is written $a\Omega = \theta^q t$ with $q \ge 1$, hence $a = \theta^{q-1}(st)^k$.
On the other hand we write $b = s^rb'$ where $b' \neq 1$ (since $\varphi (b) \not \in \bar \Theta$) and $\sigma (b') = t$.
If $r<2(q-1) + 1$, then $r = 2p+1$, $a_{2p+2} = \theta^{q-1-p} \neq 1$, $b_{2p+2} = b' \neq 1$, and $(x_{2p+1}', x_{2p+2}, y_{2p+2}) = (s,s,t)$, hence $\varepsilon = 1$.
If $r \ge 2(q-1)+1$, then $a_{2p+2} =1$, $b_{2p+2} \neq 1$ and $x_{2p+1}' = s$, hence $\varepsilon = 1$.

Suppose that $\varphi (b) \in \Theta$.
Then $\varphi(b)$ is written $\varphi(b) = \theta^q$ with $q \ge 1$, hence $b = ((ts)^k(st)^k)^q$.
On the other hand we write $a = a' s^r$ where either $a' = 1$ or $\tau (a') = t$.
We necessarily have $r=2p+1 < 2q$, $a_{2p+2} = a'$ and $b_{2p+2} = ((ts)^k(st)^k)^{q-p-1} (ts)^k$, hence $x_{2p+1}' =s$ and $x_{2p+2} = t$ if $a' \neq 1$, and therefore $\varepsilon  = 1$.

Suppose that $c \in M_1$.
If $a_{2p+2} \neq 1$ and $b_{2p+2} \neq 1$, then $x_{2p+2} = t$ and $y_{2p+2} = s$, hence $\varepsilon = 1$.
If $a_{2p+2} \neq 1$ and $b_{2p+2} = 1$, then $x_{2p+2} = t$, hence $\varepsilon = 1$.
If $a_{2p+2} = 1$ and $b_{2p+2} \neq 1$, then $y_{2p+2} = s$, hence $\varepsilon = 1$.
If $a_{2p+2} = 1$ and $b_{2p+2} = 1$, then $\varepsilon = 1$.

{\it Part (3):}
Let $\varepsilon = 1 + \varepsilon_{1,a} + \varepsilon_{3,b} - \varepsilon_{1,c} + \varepsilon_d$.
By the above we have $\dpt(c) = \dpt(a) + \dpt(b \Omega) -(p+1)(2k-1) - \varepsilon$, and $\varepsilon$ is as follows.  
If $a_{2p+2} \neq 1$ and $b_{2p+2} \neq 1$, then $\varepsilon = 0$ if $(x_{2p+1}', x_{2p+2}, y_{2p+2}) \in \{ (s,s,s), (t,t,t) \}$, and $\varepsilon = 1$ otherwise. 
If $a_{2p+2} \neq 1$ and $b_{2p+2} = 1$, then $\varepsilon = 0$ if $(x_{2p+1}',x_{2p+2}) = (s,s)$, and $\varepsilon = 1$ otherwise.  
If $a_{2p+2}=1$ and $b_{2p+2} \neq 1$, then $\varepsilon = 0$ if $(x_{2p+1}', y_{2p+2}) = (t,t)$, and $\varepsilon = 1$ otherwise. 
If $a_{2p+2}=1$ and $b_{2p+2}=1$, then $\varepsilon = 1$.

Suppose that $a \in \Theta$.
Then $a$ is written $a = \theta^q$ with $q \ge 1$. 
On the other hand we write $b = t^r b'$ where either $b' = 1$ or $\sigma (b') = s$. 
We necessarily have $r = 2p+1 < 2q$, hence $a_{2p+2} = \theta^{q-p-1} (st)^k$ and $b_{2p+2} = b'$. 
If $b' \neq 1$, then $(x_{2p+1}', x_{2p+2}, y_{2p+2}) = (t,t,s)$, hence $\varepsilon = 1$.
If $b' = 1$, then $(x_{2p+1}',x_{2p+2}) = (t,t)$, hence $\varepsilon = 1$.

Suppose that $b \Omega \in \Theta$. 
Then $b \Omega$ is written $b \Omega = \theta^q t$ with $q \ge 1$, hence $b = \theta^{q-1} (st)^k$.
On the other hand we write $a = a' t^r$ where $a'\neq 1$ (since $a \not \in \bar \Theta$) and $\tau (a') = s$.
If $r \ge 2q-1$, then $p= q-1$, $a_{2p+2} = a' t^{r-2p-1}$ and $b_{2p+2}=1$, hence $x_{2p+1}' = t$, and therefore $\varepsilon = 1$.
If $r <2q-1$, then $r =2p+1$, $a_{2p+2} = a'$ and $b_{2p+2} = \theta^{q-p-1} \neq 1$, hence $(x_{2p+1}', x_{2p+2}, y_{2p+2}) =(t,s,s)$, and therefore $\varepsilon = 1$.

Suppose that $c \in M_1$.
If $b_{2p+2} \neq 1$ and $a_{2p+2} \neq 1$, then $x_{2p+2} = t$ and $y_{2p+2} =s$, hence $\varepsilon  = 1$.
If $a_{2p+2} \neq 1$ and $b_{2p+2} = 1$, then $x_{2p+2} = t$, hence $\varepsilon = 1$.
If $a_{2p+2}=1$ and $b_{2p+2} \neq 1$, then $y_{2p+2} = s$, hence $\varepsilon = 1$.
If $a_{2p+2} = 1$ and $b_{2p+2} = 1$, then $\varepsilon = 1$.

{\it Part (4):}
Let $\varepsilon = 1 + \varepsilon_{2,a} + \varepsilon_{4,b} - \varepsilon_{2,c} + \varepsilon_d$.
By the above we have $\dpt(c \Omega) = \dpt(a \Omega) + \dpt(\varphi (b) \Omega) -(p+1)(2k-1) - \varepsilon$, and $\varepsilon$ is as follows.
If $a_{2p+2} \neq 1 $ and $b_{2p+2}\neq 1$, then: $\varepsilon = 0$ if $(x_{2p+1}', x_{2p+2}, y_{2p+2}) \in \{ (s,s,s), (t,t,t) \}$, and $\varepsilon = 1$ otherwise. 
If $a_{2p+2} \neq 1$ and $b_{2p+2} = 1$, then: $\varepsilon = 0$ if $(x_{2p+1}', x_{2p+2}) = (t,t)$, and $\varepsilon = 1$ otherwise.
If $a_{2p+2} = 1$ and $b_{2p+2} \neq 1$, then: $\varepsilon = 0$ if $(x_{2p+1}', y_{2p+2}) = (t,t)$, and $\varepsilon = 1$ otherwise.
If $a_{2p+2}=1$ and $b_{2p+2}=1$, then $\varepsilon = 0$ if $x_{2p+1}'=t$, and $\varepsilon = 1$ otherwise.

Suppose that $a \Omega \in \Theta$.
Then $a \Omega$ is written $a \Omega = \theta^q t$ with $q \ge 1$, hence $a = \theta^{q-1} (st)^k$.
On the other hand we write $b=s^r b'$ where either $b'=1$ or $\sigma (b') = t$.
If $r \ge 2q-1$, then $a_{2p+2} = 1$ and $x_{2p+1}' = s$, hence $\varepsilon = 1$.
If $r<2q-1$, then $r=2p+1$, $a_{2p+2} = \theta^{q-p-1}$ and $b_{2p+2}=b'$, hence $x_{2p+1}' = s$, $x_{2p+2}=s$ and either $b_{2p+2}=1$ or $y_{2p+2}=t$, and therefore $\varepsilon = 1$.

Suppose that $\varphi (b) \Omega \in \Theta$.
Then $\varphi (b) \Omega$ is written $\varphi (b) \Omega = \theta^q t$ with $q \ge 1$, hence $b = ((ts)^k (st)^k)^{q-1} (ts)^k$.
On the other hand we write $a = a' s^r$ where either $a'=1$ or $\tau (a') = t$.
If $r \ge 2q-1$, then $b_{2p+2} = 1$ and $x_{2p+1}' = s$, hence $\varepsilon = 1$.
If $r < 2q-1$, then $r=2p+1$, $a_{2p+2}=a'$ and $b_{2p+2} = ((ts)^k (st)^k)^{q-p-1}$, hence $x_{2p+1}'=s$ and $y_{2p+2}=t$, and therefore $\varepsilon = 1$.
\end{proof}

Now, the second part of Theorem \ref{thm6_1} is a direct consequence of the previous two lemmas.

\begin{prop}\label{prop6_11}
The pair $(H,G_1)$ satisfies Condition B with constant $\zeta  = 2k - 1$.
\end{prop}

\begin{proof}
We take two unmovable elements $a,b \in M$, and we consider the $\Delta$-form $ab = c \Delta^p$ of $ab$. 
We should prove that there exists $\varepsilon \in \{0,1\}$ such that $\dpt (c) = \dpt (a) + \dpt (b) -p (2k-1) -\varepsilon$, and $\varepsilon=1$ if either $a \in \Theta$ or $b \in \Theta$ or $c \in M_1$.
Clearly, there exist two $\Omega$-unmovable elements $a',b' \in M$ such that $(a,b) \in \{ (a',b'), (a' \Omega, \varphi(b')),
(a',b' \Omega), (a' \Omega, \varphi(b')\Omega) \}$.
Let $a'b' = d \Delta^q$ be the $\Delta$-form of $a'b'$.
Then, again, there exists an $\Omega$-unmovable element $c' \in M$ such that $d \in \{c', c'\Omega\}$.
Suppose that $d = c'$.
Then: $c=c'$ and $p=q$ if $(a,b) = (a',b')$, $c=c'\Omega$ and $p=q$ if either $(a,b) = (a' \Omega, \varphi(b'))$ or $(a,b) =(a',b' \Omega)$, and $c=c'$ and $p=q+1$ if $(a,b) = (a' \Omega, \varphi(b')\Omega)$.
These four cases are covered by Lemma \ref{lem6_9}.
Suppose that $d = c' \Omega$.
Then: $c = c' \Omega$ and $p=q$ if $(a,b) = (a',b')$, $c=c'$ and $p=q+1$ if either $(a,b) = (a' \Omega, \varphi(b'))$ or $(a,b) =(a',b' \Omega)$, and $c = c' \Omega$ and $p=q+1$ if $(a,b) = (a' \Omega, \varphi(b')\Omega)$.
These four cases are covered by Lemma \ref{lem6_10}.
\end{proof}



\end{document}